\documentclass[a4paper, 10pt]{amsart}
\usepackage[left=2.5cm, right=2.5cm, 
top=2.5cm, bottom=2.5cm]{geometry}
\usepackage[colorlinks]{hyperref}
\usepackage[english]{babel}

\usepackage{array}
\usepackage{amsmath}
\usepackage{amssymb}
\usepackage{amsthm}
\usepackage{arydshln}
\usepackage{caption}
\usepackage{comment}
\usepackage{diagbox}
\usepackage{enumerate}
\usepackage{float}
\usepackage{graphicx}
\usepackage{mathrsfs}
\usepackage{multirow}
\usepackage{setspace}
\usepackage{stmaryrd}
\usepackage{subfig}
\usepackage{wrapfig}
\usepackage{ulem}
\usepackage{xcolor}

\numberwithin{equation}{section}
\newcommand{\ds}{\displaystyle}

\newcommand{\p}{\prime}
\newtheorem{lem}{Lemma}
\newtheorem{rmk}{Remark}
\newtheorem{thm}{Theorem}
\newtheorem{ass}{Assumption}

\usepackage{etoolbox}
\BeforeBeginEnvironment{comment}{\vspace{-5pt}}
\BeforeBeginEnvironment{figure}{\vspace{-10pt}}
\AfterEndEnvironment{figure}{\vspace{-5pt}}
\BeforeBeginEnvironment{table}{\vspace{-5pt}}
\AfterEndEnvironment{table}{\vspace{-5pt}}
\BeforeBeginEnvironment{center}{\vspace{-5pt}}
\allowdisplaybreaks[4]
\title[Distributed Elliptic Control]
{A reconstructed discontinuous approximation for\\distributed elliptic control problems}

\author[R. Li]{Ruo Li} \address{CAPT, LMAM and School of Mathematical
Sciences, Peking University, Beijing 100871, P.R. China;
Chongqing Research Institute of Big Data, Peking University, Chongqing
401121, P.R. China}
\email{rli@math.pku.edu.cn}

\author[H.-Y. Liu]{Haoyang Liu}\address{National Engineering Laboratory for Big Data Analysis and Applications, Peking University,
Beijing, 100871, P.R. China. Chongqing Research Institute of Big Data, Peking University, Chongqing,
401121, P.R. China.}
\email{liuhaoyang@pku.edu.cn}

\author[J. Yin]{Jun Yin} \address{School of Mathematical
Sciences, Peking University, Beijing 100871, P.R. China}
\email{2401110056@stu.pku.edu.cn}

\begin{document}
\maketitle

\begin{abstract}
  In this paper, we present and analyze an interior penalty 
  discontinuous Galerkin 
  method for the distributed elliptic optimal control problems. 
  It is based on a reconstructed discontinuous
  approximation which admits arbitrarily high-order 
  approximation space with only one unknown
  per element.
  Applying this method, 
  we develop a proper discretization scheme that approximates 
  the state and adjoint variables in the approximation space. Our main 
  contributions are twofold: (1) the derivation of both a priori 
  and a posteriori error estimates of the $L^2$-norm and the energy 
  norms, and
  (2) the implementation of an efficiently solvable discrete system, which 
  is solved via a linearly convergent projected gradient descent method.
  Numerical experiments are provided to verify the convergence 
  order in a priori error estimate and the efficiency of a posteriori error 
  estimate.

\noindent \textbf{keywords}: distributed elliptic control; 
discontinuous Galerkin method; reconstructed discontinuous 
approximation

\end{abstract}


\section{Introduction}

Optimal control problems governed by partial differential 
equations serve as a vital bridge between advanced mathematical 
theory and real-world engineering. Their critical importance across 
applied sciences has motivated decades of research, leading to 
significant advancements in both the underlying theory and the 
numerical techniques devised to tackle them. 
We refer to \cite{lions1971ocpde} for a mathematical theory of 
optimal control 
problems and \cite{manzoni2021ocpde, teo2021ocode} for some 
numerical methods. Among 
all these methods, the finite element method is extensively used to 
numerically 
solve this class of problems due to its high accuracy and adaptability.

The numerical analysis based on the finite element method can be traced 
back to 1970s with the works in \cite{falk1973ocp, geveci1979ocp}. 
Thereafter, research on many types of 
control problems has advanced significantly over the past few decades. 
These developments range from comprehensive a priori and a posteriori 
error analysis to the creation of efficient error indicators for adaptive 
finite element methods. A description of the methodology can be 
found in \cite{hinze2008ocpde}. 
More specifically, for Dirichlet boundary control, we refer to \cite{
chowdhury2017dirichlet, 
garg2023dirichlet, gong2019adaptive}. 
For Neumann boundary control, we refer to \cite{
geveci1979ocp, casas2012neumann, 
casas2008neumann, brenner2019neumann}. 
For distributed control, we refer to \cite{
chowdhury2015priori, brenner2018distributed,  
brenner2020distributed, li2007posteriori} 
and references therein. Variational discretization approaches for the 
control constrained problems are proposed 
in \cite{hinze2005varydiscre, gong2017adaptive}. In the context of  
adaptive finite element methods, we refer to \cite{
gong2017adaptive, gong2019adaptive, li2007posteriori} and references 
therein. The
monograph \cite{liu2008adaptive} constructed a framework of 
using multi-adaptive 
meshes in developing efficient numerical algorithms for optimal 
control with error estimates. 

In recent years, discontinuous Galerkin (DG) finite element methods 
have been widely investigated; we refer to \cite{arnold2002dgunified, 
cockburn2000dgsummary, riviere2008dgelliptic} 
for some monographs. Beyond the basic feature of allowing 
inter‑element discontinuities, DG methods are distinguished by a 
combination of traits: 
inherent local conservation, compact stencils and high‑order accuracy 
with inherent parallel scalability.
These attributes make them a 
compelling choice for a wide range of problems. On the other hand, 
DG methods have relatively lower computational efficiency compared 
to continuous elements because of the significant increase in the 
number of degrees of freedom \cite{hughes2000dgcgcompare}. 
This shortcoming is particularly pronounced for high‑order schemes.

To address this issue and reduce the number of degrees of freedom, 
the authors introduced a reconstructed discontinuous approximation 
(RDA)
method in \cite{li2024RDA}. This method constructs a high-order 
approximation 
space requiring only one unknown per element. The core step involves 
solving a local least-squares problem over a patch of neighboring 
elements. The resulting space forms a small subspace of the standard 
discontinuous piecewise polynomial space, which is linearly isomorphic 
to the discontinuous piecewise constant space. Consequently, 
the number of degrees of 
freedom is substantially reduced to match the number of mesh 
elements. 
Moreover, it has been proved that optimal approximation accuracy 
can be achieved by appropriately selecting the size of the element patch. 
This approach has been successfully utilized to solve a series of practical 
problems, as documented in \cite{li2024RDA, li2024biharm}.

In this article, we focus on distributed optimal control problems 
governed by second-order elliptic equations. 
Such models are ubiquitous in physics and engineering, as they 
describe equilibrium states in diffusion and potential flow 
phenomena. A typical approach to their discretization relies on 
classical Lagrange finite elements. For instance, 
\cite{li2007posteriori} adopted this method and provided a 
priori and recovery type a posteriori estimates of the modeling error. 
In addition, \cite{brenner2018distributed, brenner2020distributed} 
studied problems with pointwise state constraints on non-convex 
domains, proving a priori error estimates using both continuous 
and discontinuous finite elements. A theoretical advancement 
concerns adaptive methods under pointwise control constraints. 
Notably, \cite{gong2017adaptive} provided a rigorous convergence 
analysis for the adaptive finite element method under variational 
control discretization. 
Furthermore, the same work established the optimality of the 
adaptive algorithm 
with respect to approximations of the state and adjoint variables.

This paper applies the aforementioned RDA method to 
this class of problems, which feature strongly convex cost functionals. 
The numerical formulation employs the approximation 
space for the state and adjoint variables, paired with a piecewise 
constant 
discretization for the control variable. We provide a well-defined 
discretization and prove both a priori and a posteriori error 
estimates of the $L^2$-norm and the energy norms. 
The discrete optimality conditions are solved using a projected 
gradient descent method, which is proved to have linear convergence 
rate. Each computational loop involves the sequential 
computation of the state and the 
adjoint equation via the RDA method described above, and updating
the control variable using the solutions to these two
equations. The 
performance and theoretical properties of this method are then 
validated through a series of numerical examples.

The rest of the article is organized as follows. In Section \ref{sec2}, 
we present 
the model problem and outline the basic structure of the RDA method. 
We put forward the discrete system and an iterative method to solve it  
using projected gradient descent. 
In Section \ref{sec3}, we present a primary 
result for a priori error estimate and a super-convergence result for the 
box constraint case with added regularity on the control variable.
In Section \ref{sec4}, we prove a general a posteriori error estimate 
and a sharp 
estimate for the box constraint case. In Section \ref{sec5}, 
we present 
some numerical examples to confirm the theoretical results. Finally, we 
summarize the article in Section \ref{sec6}.

\section{Model Problem and the RDA Method}\label{sec2}
In this section, we will describe the fundamental formulations of the 
model problem and the numerical method. In Section \ref{sec2.1}, we 
introduce the model problem. In Section \ref{sec2.2} 
and \ref{sec2.3}, we construct 
the finite element space and illustrate the RDA method. In Section \ref{sec2.4},  
a discrete system is designed to approximate the model problem. 
In Section \ref{sec2.5}, a linearly convergent iterative method 
is provided to solve the discrete system. 

\subsection{Model Problem}\label{sec2.1}
Let $ \Omega_u,\Omega\subset
\mathbb{R}^d(d=2,3)$ be convex polygonal or polyhedral  
domains. Take the control space $U=L^2(\Omega_u),$ 
the state space $V=H^2(\Omega)$ and the admissible 
set $U_{ad}\subset U$ which is closed and convex. 
Consider the following optimal control problem:
\begin{equation}\label{eq1}
\begin{aligned}
&\min_{u\in U_{ad}}g(y)+j(u),&\\
\mathrm{s.t.}\ &-\nabla\cdot(A\nabla y)=f+Bu\quad\mathrm{in}\ 
\Omega,\\
&y=\phi\quad\mathrm{on}\ \partial\Omega,
\end{aligned}
\end{equation}
where $g$ and $j$ are convex continuously 
Frech$\acute{\mathrm{e}}$t differentiable 
functionals on the observation space $C=L^2(\Omega)$ and 
the control space $U$ respectively.
The matrix $A=A(x)=(a_{ij}(x))\in (W^{1,\infty}
(\Omega))^{d\times d}$ is symmetric positive definite 
such that there exist constants $0<\lambda<\Lambda$ satisfying
\[\lambda\|\xi\|^2\leq\xi^TA(x)\xi\leq\Lambda
\|\xi\|^2\quad\forall\xi\in\mathbb{R}^d,\ x\in\Omega.\]
In addition, we assume $f\in L^2(\Omega),\phi\in H^{\frac{3}{2}} 
(\partial\Omega),B:L^2(\Omega_u)
\rightarrow L^2(\Omega)$ is a bounded linear operator. We 
further assume that 

\begin{ass}\label{ass2}
The functionals $j(u)$ and $g(y)$ 
can be written in the form
\[j(u)=\int_{\Omega_u}^{}\tilde{j}(u)\mathrm{d}x,\quad g(y)=
\int_{\Omega}^{}\tilde{g}(y)\mathrm{d}x, \]
which means that
\[j^\prime(u)(v)=\int_{\Omega_u}^{}\tilde{j}^\prime(u)v\mathrm
{d}x,\quad g^\prime(y)(w)=
\int_{\Omega}^{}\tilde{g}^\prime(y)w\mathrm{d}x,\]
for $\forall u,v\in U,y,w\in Y.$
\end{ass} 

Assumption $\ref{ass2}$ justifies the interchange of the 
Frech$\acute{\mathrm{e}}$t derivative and the integral sign, 
which is standard practice for the proof that follows. 
It is naturally satisfied by the common forms of $g(y)$ and $j(u)$, 
such as $\|y-y_d\|_{L^p(\Omega)}^p$ and $\|u-u_d\|_{L^q
(\Omega_u)}^q$ with given targets  $y_d$ and $u_d$.
Furthermore, to simplify notation, we will abuse notation by 
using $j(u)$ and $g(y)$ to refer to  $\tilde{j}(u)$ and $\tilde{g}
(y)$ respectively in the subsequent discussion.

\begin{ass}\label{ass1}
The functionals $j(u)$ and $g(y)$ 
are strongly convex and have Lipschitz continuous Frech$\acute{e}$t derivatives, i.e. we have 
\[\alpha\|u-\tilde{u}\|_{L^2}^2\leq (j^\prime(u)-j^\prime(
\tilde{u}),u-\tilde{u})_{L^2},\quad\beta\|y-\tilde{y}\|_{L^2}^2\leq
 (g^\prime(y)-g^\prime(\tilde{y}),y-\tilde{y})_{L^2},\]
\[\|j^\prime(u)-j^\prime(\tilde{u})\|_{L^2}\leq C\|u-\tilde{u}\|_
{L^2},\quad\|g^\prime(y)-g^\prime(\tilde{y})\|_{L^2}\leq C\|y-
\tilde{y}\|_{L^2},\]
for $\forall u,\tilde{u}\in U_{ad},y,\tilde{y}\in Y$. 
Here $\alpha,\beta,C$ are positive constants.
\end{ass} 

\begin{rmk}\label{rmk5}
Here we denote $\|u-\tilde{u}\|_{L^2}:=\|u-\tilde{u}\|_{
L^2(\Omega_{u})}$ and $\|y-\tilde{y}\|_{L^2}:=
\|y-\tilde{y}\|_{L^2(\Omega)}.$ We will use these notations 
in the following part if there is no confusion.
\end{rmk}

Define a bilinear form $a:H_0^1(\Omega)
\times H_0^1(\Omega)\rightarrow\mathbb{R}$ as 
\[a(y,w)
:=(A\nabla y,\nabla w)_{L^2(\Omega)}.\] To derive the optimality condition of problem $\eqref{eq1}$, we separate
the state equation into two independent parts as
\begin{flalign*}
\left\{
\begin{aligned}
&-\nabla\cdot(A\nabla y_1)=f+Bu\quad\mathrm{in}\ \Omega,\\
&y_1=0\quad\mathrm{on}\ \partial\Omega,
\end{aligned}
\right.\quad
\left\{
\begin{aligned}
&-\nabla\cdot(A\nabla y_2)=0\quad\mathrm{in}\ \Omega,\\
&y_2=\phi\quad\mathrm{on}\ \partial\Omega.
\end{aligned}
\right.
\end{flalign*}
Notice that $y_2$ does not rely on $u$, we denote ${g_1}(y_1(u))
:=g(y_1(u)+y_2)$. 
Denote the optimal solution to $\eqref{eq1}$ by $(y^\ast,
u^\ast)$, then 
\begin{equation}\label{eq6}
\left\{
\begin{aligned}
&-\nabla\cdot(A\nabla y^\ast)=f+Bu^*\quad\mathrm{in}\ 
\Omega,\\
&y^\ast=\phi\quad\mathrm{on}\ \partial\Omega,
\end{aligned}
\right.
\end{equation}
and $y_1^\ast=y^\ast-y_2$. 
Define the adjoint variable $p^\ast$ such that 
\begin{equation}\label{eq5}
\left\{
\begin{aligned}
&-\nabla\cdot(A\nabla p^\ast)=g^\prime(y^\ast)\quad
\mathrm{in}\ \Omega,\quad\;\\
&p^\ast=0\quad\mathrm{on}\ \partial\Omega.
\end{aligned}
\right.
\end{equation}

Define the cost functional 
$J(y_1(u),u):={g_1}(y_1(u))+j(u),$
we have 
 \[J(y_1(u^\ast+t(v-u^\ast)),u^\ast+t(v-u^\ast)
)\geq J(y_1^\ast,u^\ast),\quad\forall v\in U_{ad},\  t\in[0,1].\]
Differentiate the above inequality at $t=0$ in the direction $v-u^*,$ then
$$(g_1^\prime(y_1^{*})y_1^\prime(u^*)+j^\prime(u^*))
(v-u^*)\geq 0,\quad\forall v\in U_{ad}.$$
Notice that $a(y_1(v),p^*)=(f+Bv,p^*)$ for $\forall v\in U_{ad}$, we have 
\[a(y_1^\prime(u^*)(v-u^*),p^*)-(B(v-u^*),p^*)=0,\quad \forall v\in U_{ad}.\]
Here we use $(\,\cdot\,,\,\cdot\,)$ to denote the $L^2$ inner-product 
$(\,\cdot\,,\,\cdot\,)_{L^2}$. 
Since \[a(w,p^*)=(g^\prime(y^\ast),w)=(g_1^\prime
(y_1^\ast),w),\quad\forall w\in H_0^1(\Omega),\]
we take $w=y_1'(u^*)(v-u^*)\in H_0^1(\Omega)$ and derive
\begin{equation*}
(j^\prime(u^*)+B^*p^*,v-u^*)\geq0,\quad\forall v\in U_{ad}.
\end{equation*}
So far we have derived the optimality conditions of problem 
$\eqref{eq1}$ as follows.
\begin{equation}\label{eq7}
\left\{
\begin{aligned}
&-\nabla\cdot(A\nabla y^\ast)=f+Bu^*,\quad y^\ast|_{\partial
\Omega}=\phi,\\
&-\nabla\cdot(A\nabla p^\ast)=g^\prime(y^\ast),\quad p^\ast|_
{\partial\Omega}=0,\\
&(j^\prime(u^*)+B^*p^*,v-u^*)\geq0,\quad\forall v\in U_{ad}.
\end{aligned}
\right.
\end{equation}
It is well-known from \cite{lions1971ocpde} that there exists a unique 
solution 
($y^*,u^*$) to problem $\eqref{eq1}$, and the  optimality condition 
\eqref{eq7} is a necessary and sufficient condition of the exact solution ($y^*,u^*$).

\subsection{The Reconstruction Operator}\label{sec2.2}
We are going to construct the finite element space.
Let $\mathcal{T}_h$ be a quasi-uniform triangulation over 
$\Omega\subset\mathbb{R}^d$ with mesh size 
$h$. Let $\mathcal{E}_h^I$ and 
$\mathcal{E}_h^B$ be the set of all $d-1$ dimensional 
faces in the interior of $\Omega$ in $\mathcal{T}_h$, and faces 
on the boundary $\partial\Omega$ in $\mathcal{T}_h$ 
respectively. 
The set of all faces is denoted by 
 $\mathcal{E}_h:=\mathcal{E}_h^I\cup\mathcal{E}_h^B.$
Now we will briefly introduce the reconstruction 
operator $\mathcal{R}^{m}$ proposed in \cite{li2024RDA}.  The 
construction of $\mathcal{R}^{m}$ includes three steps in total.

Step 1. Given $K\in\mathcal{T}_h$, we want to construct an 
element patch $S(K)$, which consists of $K$ itself and some of its
surrounding elements. The construction of $S(K)$ is 
conducted in a recursive way. We begin
 by setting $S_0(K)=\{K\}$, and define $S_t(K)$ recursively:
$$S_t(K)=\bigcup\limits_{K'\in S_{t-1}(K)}N(K'),\quad 
t=0,1,\cdots $$ 
where $$N( K) : = \{ K^{\prime }\in \mathcal{T}_h
\mid\exists e\in\mathcal{E}_h\ \mathrm{ s.t.}\  
e\subset\partial{K^{\prime }}\cap 
\partial {K}\} .$$
The recursion stops once $t$ meets the 
condition that the cardinality $\# S_{t}( K) \geq \# S$, which 
is a given constant independent of $K$. Then
we define the patch $S( K) : = S_{t}( K) .$ 
Generally, the threshold $\# S$ can be 
chosen as $\frac{d+1}{2}\mathrm{dim}(\mathbb{P}_m(K))$ 
by numerical observations.

Step 2. For each $K\in\mathcal{T}_h$, we will solve a local 
least square problem on $S(K).$ Let
$\boldsymbol{x}_K$ be the barycenter of $K$ and 
$I(K):=\{\boldsymbol{x}_{K'}\in\Omega\mid K'\in S(K)\}$ be the set of barycenters of elements in $S(K)$.  
Let $V_h^{0}$ be the piecewise constant space  
on $\mathcal{T}_h$,
$$V_h^0:=\{w_h\in L^2(\Omega)\mid w_h|_K\in\mathbb
{P}_0(K),\:\forall K\in\mathcal{T}_h\}.$$

Given a piecewise constant function $w_h\in V_h^0$, for each 
$K\in\mathcal{T}_h$, we seek a polynomial of degree no 
more than $m$ by the following constrained local 
least squares problem:
\begin{equation}\label{eq2}
\begin{aligned}
&\underset{p\in\mathbb{P}_m(S(K))}{\operatorname*{
\min}}\sum_{\boldsymbol{x}\in I(K)}(p(\boldsymbol{x})
-w_h(\boldsymbol{x}))^2,\\
\mathrm{s.t.}\ &p(\boldsymbol{x}_K)=w_h(\boldsymbol{x}_K).
\end{aligned}
\end{equation}
The existence and uniqueness of the solution to this problem
is detailed in \cite{li2024RDA}.
We denote this solution as $w_{h,S(K)}
\in\mathbb{P}_m(S(K))$. 

Step 3.  Given $K\in\mathcal{T}_h$, notice that $w_{h,S(K)}$ 
linearly depends on the function $w_h$, we can then define a 
linear operator $\mathcal{R}_K^m:V_h^0
\longrightarrow\mathbb{P}_m(S(K))$ mapping $w_h$  to $w_{h,S(K)}$.
Finally, we define the global 
reconstruction operator $\mathcal{R}^m$ as
$$\begin{aligned}
&\mathcal{R}^{m}:\:V_h^0\longrightarrow V_h^m,\\
\mathrm{s.t.}\ &(\mathcal{R}^{m}w_h)|_{K}:=(\mathcal{R}_{K}^{m}
w_h)|_{K},\quad\forall K\in\mathcal{T}_h.
\end{aligned}$$
Here $V_h^m:=\mathcal{R}^
mV_h^0$ is the image space. $(\mathcal{R}^mw_h)|_K$ is the 
restriction of $\mathcal{R}_K^mw_h$ on $K\in\mathcal{T}_h$. 
Any piecewise constant function 
$w_h$ is mapped to a piecewise $m$-th degree polynomial 
$\mathcal{R}^mw_h$. Moreover, it can be shown 
that $\mathcal{R}^m$ is a linear isomorphism from 
$V_h^0$ to $V_h^m$, i.e. $\mathrm{dim}(V_h^m)=\mathrm{dim}
(V_h^0).$

\subsection{The Reconstructed Discontinuous Approximation Method}
\label{sec2.3}
Here we present the numerical method for the 
second-order elliptic problem based on the interior penalty DG 
method and the reconstructed space $V_h^m$. 

We first introduce some notations commonly used in the DG 
framework. Let $e\in\mathcal{E}_h^I$ be any interior face 
shared by $K^+,K^-\in\mathcal{T}_h$ with the unit 
outward normal vectors $\mathbf{n}^+,\mathbf{n}
^-$ along $e$, respectively. 
For any piecewise smooth scalar-valued function $v$ and 
vector-valued function $\boldsymbol
{\tau}$, the jump operator [$\,\cdot\,$] and the average operator 
$\{\,\cdot\,\}$ are defined as
$$[v]|_{e}:=v^{+}|_{e}\mathbf{n}^{+}+v^{-}|_{e}\mathbf
{n}^{-},\quad\{v\}|_{e}:=\frac{1}{2}(v^{+}|_{e}+v^{-}|_{e}),$$
$$[\boldsymbol{\tau}]|_{e}:=\boldsymbol{\tau}^{+}|_{e}\cdot
\mathbf{n}^{+}+\boldsymbol{\tau}^{-}
|_{e}\cdot\mathbf{n}^{-},\quad\{\boldsymbol{\tau}\}|_{e}:=
\frac{1}{2}
(\boldsymbol{\tau}^{+}|_{e}+\boldsymbol{\tau}^{-}|_{e}),$$
where $v^\pm:=v|_{K^\pm},\boldsymbol{\tau}^\pm:=\boldsymbol
{\tau}|_{K^\pm}.$ For any boundary face $e\in\mathcal{E}
_h^B$, with the unit outward normal $\mathbf{n}$ along 
$e,$ the above two operators are modified as
$$[v]|_e:=v|_e\mathbf{n},\ \{v\}|_e:=v|_e,\quad [\boldsymbol{\tau}]
 | _e:=\boldsymbol{\tau }|_e\cdot \mathbf{n},\ \{
\boldsymbol{\tau }\} | _e: = \boldsymbol{\tau }| _e.$$ 

Consider the problem
\begin{equation}\label{eq8}
\left\{
\begin{aligned}
&-\nabla\cdot(A\nabla y)=F\quad\mathrm{in}\ \Omega,\\
&y=\phi\quad\mathrm{on}\ \partial\Omega.
\end{aligned}
\right.
\end{equation}
Here $F\in L^2(\Omega),\phi\in H^{\frac{3}{2}}(\partial\Omega)$ is given. 
We seek the numerical solution $y_h^m\in V_h^m$ 
such that
\begin{equation}\label{eq9}
a_{h}(y_h^m,w_h^m)=l_{h}(w_h^m),\quad\forall 
w_h^m\in V_h^m,
\end{equation}
where $a_{h}:V_h^m\times V_h^m\rightarrow\mathbb{R},l_{h}:V_h^m
\rightarrow\mathbb{R}$ are
 given by 
$$\begin{aligned}
a_{h}(y_h^m,w_h^m)
:=&\sum_{K\in\mathcal{T}_h}\int_KA\nabla y_h^m\cdot\nabla w_h^m
\mathrm{d}x-\sum_{e\in\mathcal{E}_h}\int_e\{A\nabla y_h^m\}\cdot
[w_h^m]\mathrm{d}s\\
-&\sum_{e\in\mathcal{E}_h}\int_e\{A\nabla w_h^m\}
\cdot[y_h^m]\mathrm{d}s+\mu \sum_{e\in\mathcal{E}_h}h_e^{-1}\int_e
[y_h^m]\cdot[w_h^m]\mathrm{d}s,
\end{aligned}$$
$$l_{h}(w_h^m):=\sum_{K\in\mathcal{T}_h}
\int_KFw_h^m\mathrm{d}x-\sum_{e\in\mathcal
{E}_h^B}\int_e\{A\nabla w_h^m\}\cdot\mathbf{n}_e\phi
\mathrm{d}s+\mu \sum_{e\in\mathcal{E}_h^B}h_e^{-1}
\int_e[w_h^m]\cdot\mathbf{n}_e\phi\mathrm{d}s.$$
Here $\mu>0$ is the penalty parameter. 
Denote 
\[H^s(\mathcal{T}_h):=\{w_h\in L^2(\Omega)\mid 
w_h|_{K}\in H^s(K),\forall K\in\mathcal{T}_h\},\] 
and $\widetilde{h_e}=h_e/\mu.$ We define the following 
DG-norms 
$$\|w_h\|_{\mathrm{DG}}^2:=\sum_{K\in\mathcal{T}_h}^{}\|
\nabla w_h\|_{L^2(K)}^2+\sum_{e\in\mathcal{E}_h} \widetilde{h_e}^
{-1}\|[w_h]
\|_{L^2(e)}^2,\quad\forall w_h\in H^1(\mathcal{T}_h),$$
$$\interleave w_h\interleave_{\mathrm{DG}}^2:=\|w_h\|
_{\mathrm{DG}}^2
+\sum_{e\in\mathcal{E}_h}\widetilde{h_e}\|\{\nabla w_h\}
\|_{L^2(e)}^2,\quad\forall w_h\in H^2(\mathcal{T}_h).$$
By the inverse estimate 
and the trace inequality, we know that 
both norms are equivalent restricted on finite dimensional 
space $V_h^m$, i.e. 
$$\|w_h^m\|_{\mathrm{DG}}\leq\interleave w_h^m\interleave_
{\mathrm{DG}}\leq C\|w_h^m\|_{\mathrm{DG}},\quad\forall 
w_h^m\in V_h^m.$$ 

To enlarge the domain of the bilinear form $a_{h}$, we define
$V_h:=V_h^m+H^2(\Omega)$. The 
relationship between the energy norms and the $L^2$-norm in $V_h$  
is given in \cite{arnold1982ipdg}:
\begin{equation}\label{eqb}
\|w_h\|_{L^2}\leq C\|w_h\|_{\mathrm{DG}}
\leq C\interleave w_h\interleave_{\mathrm{DG}},\quad\forall 
w_h\in V_h.
\end{equation}
This constant $C$ only depends on $\Omega$ and the lower 
bound $K_0$ of all angles of all triangles in $\mathcal{T}_h$ and the 
grade constant $K_1$ of $\mathcal{T}_h$ defined as follows.
\[K_1:=\sup\limits_{K\in \mathcal{T}_h}\sup
\limits_{e\in\mathcal{E}_K}\frac{h_K}{h_e}.\]
Here $\mathcal{E}_K=\{e\in\mathcal{E}_h\mid e\subset \partial
 K\}.$ So it is invariant under uniform refinement 
of $\mathcal{T}_h.$ 

We can extend the domain of $a_{h},l_{h}$ to $V_h$ 
without confusion, namely $$a_{h}:V_h\times V_h
\rightarrow\mathbb{R},\quad l_{h}:V_h\rightarrow 
\mathbb{R}.$$
Subsequently, we can easily derive the stability of $a_{h}$ on 
$(H^2(\Omega)\cap H_0^1(\Omega),\|\cdot\|_{\mathrm{DG}})$ 
\[a_{h}(w_h,w_h)\geq c\|w_h\|_{\mathrm{DG}}^2,
\quad\forall w_h\in H^2(\Omega)\cap H_0^1(\Omega).\] 
This is due to the terms of  integration on every edge vanishes for 
$w_h\in H^2(\Omega)\cap H_0^1(\Omega)$. Naturally, we have the 
stability of $a_{h}$ on $(H^2(
\Omega)\cap H_0^1(\Omega),\|\cdot\|_{L^2})$ by $\eqref{eqb}$.
In addition, for the exact solution $y$ to $\eqref{eq8}$, we know by 
direct calculation that 
\begin{equation}\label{eqcon}
a_{h}(y,w_h)=l_{h}(w_h),\quad\forall w_h
\in V_h.
\end{equation}

Now we present some results for the boundedness and stability 
of $a_{h}$ on $(V_h^m,\interleave\cdot\interleave_
{\mathrm{DG}})$ and a priori and a posteriori error 
estimates for the reconstructed discontinuous approximation 
method. Their proofs refer to 
\cite{li2024RDA, riviere2003dgposteriori}.
\begin{lem}\label{lem1}
Let $a_{h}$ be 
defined with sufficiently large $\mu$, then there exist constants 
$C,c$ such that 
$$|a_{h}(y_h,w_h)|\leq C\interleave y_h\interleave_
{\mathrm{DG}} \interleave w_h\interleave_{\mathrm{DG}},
\quad\forall y_h,w_h\in V_h,$$
$$a_{h}(w_h^m,w_h^m)\geq c \interleave w_h^m
\interleave_{\mathrm{DG}}^2, \quad\forall w_h^m\in V_h^m.$$
\end{lem}

\begin{thm}\label{thm1}
Let $y\in H^{m+1}(\Omega)$ be the exact solution to 
$\eqref{eq8}$, and $y_h^m \in V_h^m$ be the numerical 
solution to $\eqref{eq9}$. Take $\mu$ 
sufficiently large, then there exists a constant 
$C=C(y)$ such that
$$\interleave y - y_h^m\interleave_{\mathrm{DG}} \leq C 
h^m \|y\|_{H^{m+1}},$$ and $$\|y - y_h^m\|_{L^2} \leq C 
h^{m+1} \|y\|_{H^{m+1}}.$$
\end{thm}
\begin{thm}\label{thm4}
Let $y\in H^{2}(\Omega)$ be the exact solution to 
$\eqref{eq8}$, and $y_h^m \in V_h^m$ be the numerical 
solution to $\eqref{eq9}$. Take $\mu$ 
sufficiently large, then there exists a constant 
$C$ such that
$$\|y - y_h^m\|_{L^2} \leq C(\eta_1(y_{h}^m)+\eta_2
(y_{h}^m)+\eta_3(y_{h}^m)),$$
where
\begin{equation}\label{eq26}
\begin{array}{c}
\ds\eta_1^2(y_{h}^m)=\sum_{K\in\mathcal{T}_h}\widetilde{h_K}^4\|
F+\nabla\cdot (A\nabla y_{h}^m)\|_{L^2(K)}^{2} ,\\
\ds\eta_2^2(y_{h}^m)  =\sum_{e\in\mathcal{E}_{h}}\widetilde{h_e}\|
[y_h^m]\|_{L^2(e)}^2 ,\quad
\eta_3^2(y_{h}^m) =\sum_{e\in\mathcal{E}_h^I}\widetilde{h_e}^3\|
[A\nabla y_{h}^m]\|_{L^2(e)}^{2}.
\end{array}
\end{equation}
Here $\widetilde{h_K}=h_K/\mu.$
\end{thm}

\subsection{Discrete System.} \label{sec2.4}
So far, we have established the numerical scheme to the second 
order elliptic equation and its convergence analysis.  Now we will
 give a proper discretization to the optimality condition 
$\eqref{eq7}$. We denote $(y^*,
u^*,p^*)$ as the exact solution to $\eqref{eq7}.$

By the consistency $\eqref{eqcon}$, we have
\begin{equation}\label{eq27}
\left\{
\begin{aligned}
&a_{h}(y^{\ast},w_h)=l_{h}[f+Bu^*,\phi]
(w_h),\quad\forall w_h\in V_h,\quad\ \ \; \\
&a_{h}(q_h,p^{\ast})=l_{h}[g^\prime(y^
{\ast}),0](q_h),\quad\forall q_h\in V_h,\\
&(j^\prime(u^*)+B^*p^{*},v-u^*)\geq 0,\quad 
\forall v\in U_{ad}.
\end{aligned}
\right.
\end{equation}
We define the discretization of $\eqref{eq27}$ as follows.
\begin{equation}\label{eq10}
\left\{
\begin{aligned}
&a_{h}(y_h^{m\ast},w_h^m)=l_{h}[f+Bu_h^*,\phi]
(w_h^m),\quad\forall w_h^m\in V_h^m,\\
&a_{h}(q_h^m,p_h^{m\ast})=l_{h}[g^\prime(y_h^
{m\ast}),0](q_h^m),\quad\forall q_h^m\in V_h^m,\\
&(j^\prime(u_h^*)+B^*p_h^{m*},v_h-u_h^*)\geq 0,\quad 
\forall v_h\in U_{adh}.
\end{aligned}
\right.
\end{equation}
Here for $w_h\in V_h,$ 
$$l_{h}[s,\phi](w_h):=\sum_{K\in\mathcal{T}_h}
\int_Ksw_h\mathrm{d}x-\sum_{e\in\mathcal{E}
_h^B}\int_e\{A\nabla w_h\}\cdot\mathbf{n}_e\phi
\mathrm{d}s+\mu\sum_{e\in\mathcal{E}_h^B}h_e^{-1}
\int_e[w_h]\cdot\mathbf{n}_e\phi\mathrm{d}s,$$
and $U_{adh}$ is a discretization of $U_{ad}$. 
Specifically, given a quasi-uniform mesh $\mathcal{T}_h^u$ 
on $\Omega_u$ with mesh size $h_u,$ 
we take $U_h$ to be the piecewise constant 
function space on $\mathcal{T}_h^u$ and define $U_{adh}:=U_{ad}
\cap U_h$, as in \cite{liu2008adaptive, li2007posteriori, yang2008bilinear}.

\begin{rmk}\label{rmk1.2}
Correspondingly, we can consider the discrete optimal control 
problem
\begin{equation}\label{eq11}
\begin{aligned}
&\min_{u_h\in U_{adh}}g(y_h^m)+j(u_h),&\\
\mathrm{s.t.}\ &\ a_{h}(y_h^{m},w_h^m)=l_{h}[f+Bu_h,\phi]
(w_h^m),\quad\forall w_h^m\in V_h^m,\\
&\ y_h^m\in V_h^m.
\end{aligned}
\end{equation}
Take $\mu$ sufficiently large, we shall prove that the discrete 
system $\eqref{eq10}$ is exactly 
the optimality condition of the discrete problem $\eqref{eq11}$. 
Then the 
finite dimensionality of  $\eqref{eq11}$ implies the existence of 
the solution to $\eqref{eq10}$. The strong convexity in 
Assumption $\ref{ass1}$ guarantees the uniqueness of 
the solution to $\eqref{eq10}$.
\begin{proof}
Let $(y_h^{m*},u_h^*)$ be the solution to \eqref{eq11}. 
Notice that  
 \[J(y_h^m(u_h^\ast+t(v_h-u_h^\ast)),u_h^\ast+t(v_h-u_h^\ast)
)\geq J(y_h^{m*},u_h^*),\quad\forall v_h\in U_{adh},\  t\in[0,1].\]
Differentiate the above inequality at $t=0$ in the direction $v_h-
u_h^*,$ we have
\begin{equation*}
\begin{aligned}
&(\tilde{g}^\prime(y_h^{m*})y_h^{m\prime}(u_h^*)+\tilde{j}^
\prime(u_h^*),v_h-u_h^*)\geq 0,\quad \forall v_h\in U_{adh}.
\end{aligned}
\end{equation*}
Here we use Assumption $\ref{ass2}$. 
Denote $F_h(w,q)$ to be the numerical solution to the following 
equation under scheme $\eqref{eq9}.$
\begin{equation*}
\left\{
\begin{aligned}
&-\nabla\cdot(A\nabla y)=w\quad\mathrm{in}\ \Omega,\\
&y=q\quad\mathrm{on}\ \partial\Omega.
\end{aligned}
\right.
\end{equation*} 
Using the equality 
\[(\tilde{g}'(y_h^{m*}),y_h^m(v_h)-F_h(f+Bv_h,\phi))=0,
\quad\forall v_h\in U_{adh},\]
we deduce that 
\begin{equation*}
(\tilde{g}'(y_h^{m*}),y_h^{m\prime}(u_h^*)(v_h-u_h^*)-F_h(B(v_h-u_h^*),0))=0,
\quad\forall v_h\in U_{adh}.
\end{equation*}

In addition, for $\forall w\in L^2(\Omega),$ 
\begin{equation*}
\begin{aligned}
&(\tilde{g}'(y_h^{m*}),F_h(w,0))=l_{h}[\tilde{g}'(y_h^{m*}),0](F_h(w,0))\\=\ 
&a_{h}(F_h(\tilde{g}'(y_h^{m*}),0),F_h(w,0))=a_{h}(F_h(w,0),F_h(\tilde{g}'
(y_h^{m*}),0))\\=\ &l_{h}[w,0](F_h(\tilde{g}'(y_h^{m*}),0))=(F_h(\tilde{g}'
(y_h^{m*}),0),w).
\end{aligned}
\end{equation*}
Take $w=B(v_h-u_h^*)$ to derive 
\[(\tilde{j}^\prime(u_h^*)+B^*F_h(\tilde{g}'(y_h^{m*}),0),v_h-u_h^*)\geq 0,\quad\forall v_h\in U_{adh}.\]
Define $p_h^{m*}:=F_h(\tilde{g}'(y_h^{m*}),0)$ and 
we can derive the  second and third equation in \eqref{eq10}. Hence 
$\eqref{eq10}$ is the optimality condition of $\eqref{eq11}.$
\end{proof}
\end{rmk}

\subsection{Iterative Method}\label{sec2.5}
It should be noted that the approximate system $\eqref{eq10}$ 
can not be solved directly.
Here we introduce an iterative method. 
Take $u_{h0}\in U_{adh}$ as an  
initial guess, we can then obtain a sequence \{($u_{hn},y_
{hn}^m,p_{hn}^m)\}_{n=0} ^\infty$ by the following algorithm.
\begin{equation}\label{eq5.1}
\left\{
\begin{aligned}
&a_{h}(y_{hn}^{m},w_{h}^m)=l_{h}[f+Bu_{hn}
,\phi](w_{h}^m),\quad\forall w_{h}^m\in V_h^m,\\
&a_{h}(q_h^m,p_{hn}^{m})=l_{h}[g^\prime(y_{hn}^
{m}),0](q_h^m),\quad\forall q_h^m\in V_h^m,\\
&u_{h \mkern+1mu n \mkern-2mu+\mkern-2mu 1}=Pr_{U_{adh}}
(u_{hn}-\rho_n(j'(u_{hn})+B^*p_{hn}^m)),
\end{aligned}
\right.
\end{equation}
where $Pr_{U_{adh}}:U\rightarrow U_{adh}$ is the projection 
operator and $\rho_n>0,\forall n\in \mathbb{N}.$ This is essentially 
a projected gradient descent method, which has linear convergence 
rate under strongly convex cost function. We will prove this property 
in detail.
\begin{lem}\label{lem2.2}
Let $(u_h,y_h^{m},p_h^{m})$ satisfy $\eqref{eq10}$, $\{(u_{hn},
y_{hn}^m,p_{hn}^m)\}_{n=0}^\infty$ be the sequence generated by 
$\eqref{eq5.1}$. If Assumptions $\ref{ass2}$,$\ref{ass1}$ 
hold, we can properly  
choose $\{\rho_n\}_{n=0}^\infty$ such that there exist constants
$0<\eta<1,\lambda=\lambda(\eta)>0,$ satisfy
\begin{equation*}
\begin{aligned}
&\|u_{hn}-u_h\|_{L^2}+\lambda\interleave y_{hn}
^m-y_h^{m}\interleave_{\mathrm{DG}}+\lambda\interleave p_{n}^m-
p_h^{m}\interleave_{\mathrm{DG}}\\
\leq\ & \eta^n(\|u_{h0}-u_h\|_{L^2}+\lambda\interleave y_{h0}^m-y_h^{m}
\interleave_{\mathrm{DG}}+\lambda\interleave
 p_{h0}^m-p_h^{m}\interleave_{\mathrm{DG}}).
\end{aligned}
\end{equation*}
In fact, we can show that $(u_{hn},y_{hn}^m,p_
{hn}^m)$ converges to $(u_{h},y_{h}^{m},p_{h}^{m})$ linearly. 
\end{lem}

\begin{proof}
We have
\begin{equation*}
\begin{aligned}
\|u_{h\mkern+1mu n\mkern-2mu+\mkern-2mu 1}-u_h\|_{L^2}^2=
&\|Pr_{U_{adh}}(u_{hn}-\rho_n(j'(u_{hn})+B^*p_{hn}^m))\\
&-Pr_{U_{adh}}(u_{h}-\rho_n(j'(
u_{h})+B^*p_{h}^m))\|_{L^2}^2\\
\leq&\|u_{hn}-u_{h}-\rho_n(j'(u_{hn})-j'(
u_{h})+B^*p_{hn}^m-B^*p_{h}^m)\|_{L^2}^2\\
=&\|u_{hn}-u_{h}\|_{L^2}^2+\rho_n^2\|j'(u_{hn})-j'(
u_{h})+B^*p_{hn}^m-B^*p_{h}^m\|_{L^2}^2\\
&-2\rho_n(u_{hn}-u_h,j'(u_{hn})-j'(u_h)+B^*p_{hn}^m-B^*p_h^m)\\
\leq&(1-2\alpha\rho_n+C\rho_n^2)\|u_{hn}-u_{h}\|_{L^2}^2+
C\rho_n^2\|p_{hn}^m-p_h^m\|_{L^2}^2\\
&-2\rho_n(u_{hn}-u_h,B^*p_{hn}^m-B^*p_h^m).
\end{aligned}
\end{equation*} 

From $\eqref{eq10}$ and $\eqref{eq5.1}$ we have
\begin{equation}\label{eq5.2}
\begin{aligned}
&a_{h}(y_{hn}^m-y_h^m,w_h^m)=(B(u_{hn}-u_h),w_h^m),\quad\forall w_h^m
\in V_h^m,\\
&a_{h}(q_h^m,p_{hn}^m-p_h^m)=(g^\p(y_{hn}^m)-g^\p(y_h^m),q_h^m),\quad
\forall q_h^m\in V_h^m.
\end{aligned}
\end{equation}
Let $w_h^m=p_{hn}^m-p_h^m,q_{hn}^m=y_{hn}^m-y_h^m.$ Then
\begin{equation*}
\begin{aligned}
&(u_{hn}-u_h,B^*p_{hn}^m-B^*p_h^m)=(B(u_{hn}-u_h),p_{hn}^m-
p_h^m)\\
=\ &(g^\p(y_{hn}^m))-g^\p(y_h^m),y_{hn}^m-y_h^m)\geq
\beta\|y_{hn}^m-y_h^m\|_{L^2}^2\geq0.
\end{aligned}
\end{equation*}
Let $w_h^m=y_{hn}^m-y_h^m,q_{hn}^m=p_{hn}^m-p_h^m.$ By the 
stability of $a_{h}$ on $(V_h^m,\interleave\cdot\interleave
_{\mathrm{DG}})$ and $\eqref{eqb}$, we have
\[\interleave y_{hn}^m-y_h^m\interleave_{\mathrm{DG}}\leq C
\|u_{hn}-u_h\|_{L^2},
\quad\interleave p_{hn}^m-p_h^m\interleave _{\mathrm{DG}}\leq C
\|y_{hn}^m-y_h^m\|_{L^2}.\]

Consequently,
\[\|u_{h \mkern+1mu n \mkern-2mu+\mkern-2mu 1}-u_h\|_{L^2}^2
\leq(1-2\alpha\rho_n+C\rho_n^2)\|u_{hn}-u_h\|_{L^2}^2\leq
\delta^2\|u_{hn}-u_h\|_{L^2}^2,\]
where $\rho_n$ satisfies $0\leq1-2\alpha\rho_n+C\rho
_n^2\leq\delta^2<1,\forall n\in\mathbb{N}.$ 
Choose $C'$ such that
\[\interleave y_{h \mkern+1mu n \mkern-2mu+\mkern-2mu 1}^m-y_h
^m\interleave_{\mathrm{DG}}+\interleave p_{h \mkern
+1mu n \mkern-2mu+\mkern-2mu 1}^m-p_h
^m\interleave_{\mathrm{DG}}\leq C'\|u_{h \mkern+1mu n \mkern
-2mu+\mkern-2mu 1}-u_h\|_{L^2},\]
and $\lambda=(1-\delta)/(2C'\delta)$. It can be verified that
\begin{equation*}
\begin{aligned}
&\|u_{h \mkern+1mu n \mkern-2mu+\mkern-2mu 1}-u_h\|_{L^2}
+\lambda\interleave y_{h \mkern+1mu n \mkern-2mu+\mkern-2mu 1}
^m-y_h^m\interleave_{\mathrm{DG}}+\lambda\interleave p_{h 
\mkern+1mu n \mkern-2mu+\mkern-2mu 1}^m-p_h
^m\interleave_{\mathrm{DG}}\\
\leq\ & \frac{\delta+1}{2}(\|u_{hn}-u_h\|_{L^2}+\lambda\interleave y_{hn}
^m-y_h^m\interleave_{\mathrm{DG}}+\lambda
\interleave p_{hn}^m-p_h^m\interleave_{\mathrm{DG}}),
\end{aligned}
\end{equation*}
which leads to the linear convergence rate. To complete the proof, 
we only need to take $\eta=(\delta+1)/2$.
\end{proof}

With the help of Lemma $\ref{lem2.2}$, we can find a numerical 
solution $(y_{hn}^m,u_{hn},p_{hn}^m)$ close enough to the solution 
$(y_h^{m*},u_h^*,p_h^{m*})$ to $\eqref{eq10}$. 
It remains to figure out 
the modeling error of $\eqref{eq10}$ with respect to $\eqref{eq27},$ 
which will be done in the next two sections.
Throughout the following sections, we adopt the notation 
$(y^*,u^*,p^*)$ 
and $(y_h^{m*},u_h^*,p_h^{m*})$ for solutions to 
\eqref{eq27} and \eqref{eq10}, respectively. However, in the 
statement and proof of theorems, we will simply use $(y,u,p)$ 
and $(y_h^m,u_h,
p_h^m)$ whenever no confusion arises.
\section{A Priori Estimates}\label{sec3}
In this section, we will establish a priori error estimate of the 
discrete system $\eqref{eq10}$ with respect to the optimality 
condition $\eqref{eq27}$. We will first prove a theorem 
in general case, and 
then improve the result to a sharper estimate for the box 
constraint case.

\subsection{A General Result}\label{sec3.1} To complete the foundation for 
the proof below, we introduce one further technical assumption, 
which will be employed in the following part.

\begin{ass}\label{ass3} 
Let $\Pi_h^u:U\rightarrow U_h$ 
be the orthogonal projection to $U_h$ under the $L^2$-norm, 
such that $(u-\Pi_h^u u,v_h)=0,\forall v_h\in U_h.$ We shall assume that
$$\Pi_h^uU_{ad}\subset{U_{adh}},\quad \mathrm{i.e.}\ \Pi_h^uv\in U_{adh},\forall v\in 
U_{ad}.$$ 
\end{ass}

It should be noted that Assumption $\ref{ass3}$ is available 
for the following cases:
\[U_{ad}=U,\quad U_{ad}=\{v\in U\mid\gamma\leq\int_{\Omega_u}v
\mathrm{d}x\leq\Gamma\},\quad U_{ad}=\{v\in U\mid\gamma\leq v
\leq\Gamma\ \mathrm{a.e.\ in}\ \Omega_u\},\]
where $\gamma\leq\Gamma$ are constants.

\begin{thm}\label{thm2}
Let $(y,u,p)$ satisfy $\eqref{eq27},$ $(y_h^{m},u_h
,p_h^{m})$ satisfy $\eqref{eq10}$. If Assumptions 
\ref{ass2},\ref{ass1},\ref{ass3} hold, 
$u,j^\prime(u)+B^*p$ 
are Lipschitz continuous, and $y,p\in H^{m+1}(\Omega),$ 
then there exists a constant $C=C(y,u,p)$ such that
\[\|u-u_h\|_{L^2}+
\interleave y-y_h^{m}\interleave_{\mathrm{DG}}+\interleave p-p_h^{m}\interleave_{\mathrm{DG}}\leq C(h_u+h^m),\]and
\[\|u-u_h\|_{L^2}+
\| y-y_h^{m}\|_{L^2}+\| p-p_h^{m}\|_{L^2}\leq C(h_u+h^{m+1}).\]
\end{thm}

\begin{proof}
We construct the intermediate states as follows.
\begin{equation}\label{eq13}
\begin{aligned}
&a_{h}(P_hy^{},w_h^m)=l_{h}[f+Bu,\phi](w_h^m),
\quad\forall w_h^m\in V_h^m,\\
&a_{h}(q_h^m,P_hp)=l_{h}[g^\prime(y),0](q_h^m),
\quad\forall q_h^m\in V_h^m.
\end{aligned}
\end{equation}
$P_hy,P_hp\in V_h^m.$ From $\eqref{eq10}$ and $\eqref{eq13}$, 
we have
\begin{equation}\label{eq14}
\begin{aligned}
&a_{h}(P_hy-y_h^m,w_h^m)=(B(u-u_h),w_h^m),\quad\forall w_h^m\in V_h^m,\\
&a_{h}(q_h^m,P_hp-p_h^m)=(g^\prime(y)-g^\prime(y_h^m),q_h^m),\quad
\forall q_h^m\in V_h^m.
\end{aligned}
\end{equation}
Take $w_h^m=P_hp-p_h^m,q_h^m=P_h y-y_h^m$ in $\eqref{eq14}$ 
and subtract the resulting equations to find
$$(B^*(P_hp-p_h^m), u-u_h)=(g^\prime(y)-g^\prime(y_h^m),P_hy-
y_h^m).$$
This implies
\begin{equation}\label{eq15}
(B^*(P_hp-p_h^m), u-u_h)\geq\beta\|P_hy-y_h^m\|^2_{L^2}+
(g^\prime(y)-g^\prime(P_hy),P_hy-y_h^m).
\end{equation}

From the inequalities in $\eqref{eq27}$ and $\eqref{eq10}$, 
we know that $(j^\prime(u)+B^*p,u_h-u)\geq 0$ and 
$$(j^\prime(u_h)+B^*p_h^m,u-u_h)\geq (j^\prime(u_h)+
B^*p_h^m,u-v_h),\quad \forall v_h\in U_{adh}.$$ 
Thus
$$(j^\prime(u_h)-j^\prime(u)+B^*(p_h^m-p),u-u_h)\geq(j^\prime
(u_h)+B^*p_h^m, u - v_h) ,\quad\forall v_h \in U_{adh},$$
which implies
\begin{equation}\label{eq16}
( B^*(p_h^m - p), u - u_h) \geq\alpha \|u - u_h\|_{L^2}^2 + 
( j^\prime(u_h)+B^*p_h^m, u - v_h) ,\quad\forall v_h\in U_{adh}.
\end{equation}
Take $v_h=\Pi_h^u u$ (by Assumption \ref{ass3}).

Combine $\eqref{eq15}$ and $\eqref{eq16}$, we have
$$\begin{aligned}
&\alpha\|u-u_{h}\|_{L^2}^{2}+\beta\|y_{h}^m-P_{h}y\|_{L^2}^{2}\\
\leq&-(j^\prime(u_h)+B^*p_h^m,u-\Pi_h^uu)-( B^*(p-P_{h}p),
u-u_{h}) +(g^\p(y)-g^\p(P_hy),y_h^m-P_hy) \\
  =&-(j^\prime(u)+B^*p,u-\Pi_h^uu) -(j^\prime(u_h)-j^\p(u)+B^*
(p_h^m-p),u-\Pi_h^uu)\\
& -( B^*(p-P_{h}p),u-u_{h}) 
 +(g^\p(y)-g^\p(P_hy),y_h^m-P_hy) 
\end{aligned}$$
$$\begin{aligned}
  =&-( (I-\Pi_h^u)(j^\prime(u)+B^*p),u-\Pi_h^uu) -( j^\prime
(u_h)-j^\p(u)+B^*(p_h^m-p),u-\Pi_h^uu) \\
& -( B^*(p-P_hp),u-u_h)+(g^\p(y)-g^\p(P_hy),y_h^m-P_hy)\\
 \leq
& C(\|(I-\Pi_h^u)(j^\prime(u)+B^*p) \|_{L^2}^{2}+\|u-\Pi_h^uu\|
_{L^2}^{2}+\|B^*(p_h^m-p)\|_{L^2}
\|u-\Pi_h^uu\|_{L^2}\quad\quad\ \ \\
&+\|B^*(p-P_hp)\|_{L^2}^2+\|y-P_hy\|_{L^2}^{2})+\frac{\alpha}{2}
\|u-u_{h}\|_{L^2}^{2}+\frac{\beta}{2}\|y_{h}^m-P_{h}y\|_{L^2}^{2}.\\
\end{aligned}$$

Now we give an estimate for $\|p_h^m-p\|_{L^2}.$ Let 
$q_h^m=P_hp-p_h^m$ in $\eqref{eq14},$ by the stability of 
$a_{h}$ on $(V_h^m,\interleave
\cdot\interleave_{\mathrm{DG}})$, we have $$\|P_hp-
p_h^m\|_{L^2}\leq C\interleave P_h p-p_h^m\interleave_{\mathrm
{DG}}\leq C\|y-y_h^m\|_{L^2}.$$
Thus
\begin{equation*}
\begin{aligned}
\|p-p_h^m\|_{L^2}= &\|p-P_h p+P_h p-p_h^m\|
_{L^2}\leq C(\|y-y_h^m\|_{L^2}+\|p-P_h p\|_{L^2})\\
\leq&C(\|y-P_h y\|_{L^2}+\|y_h^m-P_h y\|_{L^2}+\|p-P_hp
\|_{L^2}).
\end{aligned}
\end{equation*}
So we have 
$$\begin{aligned}
&\alpha\|u-u_{h}\|_{L^2}^{2}+\beta\|y_{h}^m-P_{h}y\|_
{L^2}^{2}\\
\leq& C(\|(I-\Pi_h^u)(j^\prime(u)+B^*p)\|_{L^2}^{2}+
\|u-\Pi_h^uu\|_{L^2}^{2}
+\|p-P_hp\|_{L^2}^2+\|y-P_{h}y\|_{L^2}^{2}) \\ 
&+\frac{\alpha}{2}
\|u-u_{h}\|_{L^2}^{2}+\frac{3\beta}{4}\|y_{h}^m-P_{h}y\|_
{L^2}^{2},
\end{aligned}$$
which means that
\begin{equation*}
\begin{aligned}
&\|u-u_{h}\|_{L^2}+\|y-y_h^m\|_{L^2} \\
\leq &C(\|(I-\Pi_h^u)(j^\prime(u)+B^*p)\|_{L^2}+\|u-\Pi_h^uu\|
_{L^2}+\|p-P_hp\|_{L^2}+\|y-P_{h}y\|_{L^2}) 
\end{aligned}
\end{equation*}
by the triangular inequality.

Now we give upper bounds for $\interleave y-y_h^m
\interleave_{\mathrm{DG}}$ and $\interleave p-p_h^m\interleave
_{\mathrm{DG}}.$ We have known that 
\begin{equation*}
\interleave p-p_h^m\interleave_{\mathrm{DG}}
\leq C(\interleave p-P_h p\interleave
_{\mathrm{DG}}+
\|y-y_h^m\|_{L^2}).
\end{equation*}
For $\interleave y-y_h^m\interleave_{\mathrm{DG}}$, again we 
use the stability of $a_{h}$ on $(V_h^m,\interleave
\cdot\interleave_{\mathrm{DG}})$. Let $w_h^m=P_h y-y_h^m$ in 
$\eqref{eq14},$ we have $\interleave P_h y-y_h^m \interleave_
{\mathrm{DG}}\leq C\|u-u_h\|_{L^2}.$ Thus
\begin{equation*}
\interleave y-y_h^m\interleave_{\mathrm{DG}}
\leq C(\interleave y-P_hy
\interleave_{\mathrm{DG}}+\|u-u_h\|_{L^2}).
\end{equation*}
Combine the above inequalities, we have
\begin{equation*}
\begin{aligned}
&\|u-u_{h}\|_{L^2}+\interleave y-y_h^m\interleave_{\mathrm{DG}}
+\interleave p-p_h^m\interleave_{\mathrm{DG}} \\
\leq &C(\|(I-\Pi_h^u)(j^\prime(u)+B^*p)\|_{L^2}+\|u-\Pi_h^uu\|
_{L^2}+\interleave p-P_h p\interleave_{\mathrm{DG}}+\interleave 
y-P_{h} y\interleave_{\mathrm{DG}}).
\end{aligned}
\end{equation*}
Notice that if we substitute all the $\interleave\cdot\interleave_{\mathrm
{DG}}$ by $\|\cdot\|_{L^2}$, this inequality still holds.

For Lipschitz continuous function $v\in U,$ $\|v-\Pi_h
^u v\|\leq Ch_u$. Moreover, by Theorem 
$\eqref{thm1}$, we have \[\interleave p-P_hp\interleave_{\mathrm
{DG}}+\interleave y-P_{h} y\interleave_{\mathrm{DG}}\leq 
Ch^{m},\] and also 
\[\|p-P_hp\|_{L^2}+\|y-P_hy\|_{L^2}\leq Ch^{m+1}.\] 
Then the proof is completed. 
\end{proof}
\subsection{Super-convergence Analysis.}\label{sec3.2}
We mention here that the order of $h_u$ on the right-hand side of 
Theorem \ref{thm2} can be improved to ${3}/{2}$ if we 
concentrate on the case where
\[U_{ad}=\{v\in L^2({\Omega}_u)\mid v\geq\gamma\ \mathrm
{a.e.\ in}\ {\Omega}_u\},\]
where $\gamma$ is a constant. To derive the higher order 
convergence, we need more assumptions on regularity as follows. 
Consider two disjoint subsets of 
$\Omega_u$.
\begin{equation*}
\begin{aligned}
{\Omega}_{u}^+ =\{K_u \in \mathcal{T}_h^u\mid u^*|_
{ K_u}>\gamma\},\quad
{\Omega}_{u}^0=\{K_u \in \mathcal{T}_h^u\mid u^*|_{ 
K_u}=\gamma\}.
\end{aligned}
\end{equation*}
Define $\Omega_u^b:=\Omega_u\setminus(\Omega_u^+\cup\Omega_u^0).$ 
\begin{ass}\label{ass4}
The measure of $\Omega_u^b$ such that $m(\Omega_u^b)
\leq Ch_u,$ where $C>0$ is a constant.
\end{ass}
Assumption $\ref{ass4}$ can be easily achieved for the 
case that the exact solution $u^*$ has the form $u^*=\max\{w,\gamma\}$ with the function 
$w$ smooth enough. For sufficiently smooth $w$, the set 
$\{w=\gamma\}$ would be 
a smooth rectifiable curve, which results in $m(\Omega_u^b)\leq 
Ch_u$ by the quasi-uniformity of  $\mathcal{T}_h^u.$

Consider the space $H^1(\mathcal{T}_h^u)$ equipped with  
the energy norm $\|\cdot\|_{\mathrm{DG}}^u$, 
which has the same 
expression as the norm $\|\cdot\|_{\mathrm{DG}}$ 
on $H^1(\mathcal{T}_h)$. We want to 
add more regularity to the operator 
$B^*:L^2(\Omega)\rightarrow L^2(\Omega_u).$
\begin{ass}\label{ass5}
The operator $B^*$ satisfies 
$B^*V_h^m\subset H^1(\mathcal{T}_h^u)$ and $B^*|_{V_h^m}:
V_h^m\rightarrow H^1(\mathcal{T}_h^u)$ is bounded.
\end{ass}
Now we use these two assumptions to prove the super-close property 
of $u_h^*$ and $\Pi_h^uu^*$.

\begin{lem}\label{lem2}
Let $(y,u,p)$ satisfy $\eqref{eq27},$ $(y_h^{m},u_h
,p_h^{m})$ satisfy $\eqref{eq10}$. If Assumptions 
\ref{ass2},\ref{ass1},\ref{ass4},\ref{ass5} 
hold, $u,j^\prime(u)+B^*p \in W^{1,\infty}(\Omega_u),
y,p\in H^{m+1}(\Omega),$ 
then there exists a constant $C=C(y,u,p)$ such that
\[\|u_h-\Pi_h^uu\|_{L^2}\leq C(h_u^{\frac{3}{2}}+h^{m+1}).\]
\end{lem}
\begin{proof}
Consider the intermediate 
state as follows.
\begin{equation}\label{eq3.1}
\begin{aligned}
&a_{h}(y_\Pi^{m},w_h^m)=l_{h}[f+B\Pi_h^uu,\phi]
(w_h^m),\quad\forall w_h^m\in V_h^m,\\
&a_{h}(q_h^m,p_\Pi^m)=l_{h}[g^\prime(y_\Pi^m),0]
(q_h^m),\quad\forall q_h^m\in V_h^m.
\end{aligned}
\end{equation}
$y_\Pi^m,p_\Pi^m\in V_h^m.$ From $\eqref{eq10}$ and 
$\eqref{eq3.1}$, we have
\begin{equation}\label{eq3.2}
\begin{aligned}
&a_{h}(y_\Pi^m-y_h^m,w_h^m)=(B(\Pi_h^uu-u_h),w_h^m),
\quad\forall w_h^m\in V_h^m,\\
&a_{h}(q_h^m,p_\Pi^m-p_h^m)=(g^\prime(y_\Pi^m)-g^
\prime(y_h^m),q_h^m),\quad\forall q_h^m)\in V_h^m.
\end{aligned}
\end{equation}
Take $w_h^m=p_\Pi^m-p_h^m,q_h^m=y_\Pi^{m}-y_h^m$ in $\eqref
{eq3.2}$ and subtract the resulting equations to find
$$(B^*(p_\Pi^m-p_h^m), \Pi_h^uu-u_h)=(g^\prime(y_\Pi^{m})-g^
\prime(y_h^m),y_\Pi^{m}-y_h^m)\geq0.$$
Therefore, we have
\begin{equation*}
\begin{aligned}
 \alpha\|u_h-\Pi_h^uu\|_{L^2}^2
 \leq& (j^{\prime}(u_h)+B^*p_h^m,u_h-\Pi_h^uu)-(j^{\prime}(
\Pi_h^uu)+B^*p_\Pi^m,u_h-\Pi_h^uu) \\
 =&(j^{\prime}(u_h)+B^*p_h^m,u_h-\Pi_h^uu)+(j^{\prime}(u)+
B^*p,\Pi_h^uu-u_h) \\
 & +(j^{\prime}(\Pi_h^uu)-j^{\prime}(u),\Pi_h^uu-u_h)+(B^*p_\Pi
^m-B^*p,\Pi_h^uu-u_h) \\
 \leq& 0+(j^\prime(u)+B^*p,
u-u_h) +(j^\prime(u)+B^*p,\Pi_h^uu-u) \\
 &+(j^{\prime}(\Pi_h^uu)-j^{\prime}(u),\Pi_h^uu-u_h)+(B^*p_\Pi
^m-B^*p,\Pi_h^uu-u_h)\\
 \leq&0+0+I_1+I_2+I_3.
\end{aligned}
\end{equation*}
Here $I_1,I_2,I_3$ are the last three terms respectively. 

By the results in \cite{liu2008adaptive}(Lemma 4.3.1), we know that under Assumption 
$\ref{ass4}$, 
\[I_1\leq Ch_u^3,\quad I_2\leq Ch_u^2\|\Pi_h^uu-u_h\|_{L^2},\]
where $C$ depends on $\|u\|_{W^{1,\infty}}$ and 
$\|j'(u)+B^*p\|_{W^{1,\infty}}.$
Next we will give upper bounds of $\|y_\Pi^m-y\|_{L^2}$ and 
$\|p_\Pi^m-p\|_{L^2}$ to estimate $I_3$.

Denote $\|\cdot\|_{-\mathrm{DG}}^u$  as the dual norm of 
$\|\cdot\|_{\mathrm{DG}}^u$ 
with respect to the $L^2$-inner product ($\,\cdot\,,\,\cdot\,)$.
\begin{equation*}
\begin{aligned}
\|\Pi_h^uu-u\|_{-\mathrm{DG}}^u&:=
\sup_{v\in H^1(\mathcal{T}_h^u)}
\frac{(\Pi_h^uu-u,v)}{\|v\|_{\mathrm{DG}}^u}=\sup_{v\in H^1(
\mathcal{T}_h^u)}
\frac{(\Pi_h^uu-u,v-\Pi_h^uv)}{\|v\|_{\mathrm{DG}}^u}\\
&\;\leq \sup_{v\in H^1(\mathcal{T}_h^u)}\frac{C\sum_
{\tau_u\in\mathcal{T}_h^u}h_{\tau_u}^2|u|_{1,\tau_u}|v|_{1,\tau_u}}
{\|v\|_{\mathrm{DG}}^u}
\leq Ch_u^2\|u\|_{H^1}.
\end{aligned}
\end{equation*}
Recall the definitions of $P_hy,P_hp$ in Theorem $\ref{thm2}$, 
we have
\begin{equation*}
\begin{aligned}
&a_{h}(y_\Pi^m-P_hy,y_\Pi^m-P_hy)=(B(\Pi_h^uu-u),y_
\Pi^m-P_hy),\\
&a_{h}(p_\Pi^m-P_hp,p_\Pi^m-P_hp)=(g^\prime
(y_\Pi^m)-g^\prime(y),p_\Pi^m-P_hp).
\end{aligned}
\end{equation*}
By the stability of $a_{h}$ on $(V_h^m,\|\cdot\|_{\mathrm
{DG}})$, 
\begin{equation*}
c\|y_\Pi^m-P_hy\|_{\mathrm{DG}}^2\leq a_{h}(y_\Pi^m-
P_hy,y_\Pi^m-P_hy)\leq \|\Pi_h^uu-u\|
_{-\mathrm{DG}}^u\|B^*(y_\Pi^m-P_hy)\|_{\mathrm{DG}}^u.
\end{equation*}
Using Assumption $\ref{ass5}$, we have 
\[\|y_\Pi^m-P_hy\|_{\mathrm{DG}}\leq C\|\Pi_h^uu-u\|_{-\mathrm
{DG}}^u\leq Ch_u^2\|u\|_{H^1}.\]
By $\eqref{eqb}$, we know that
\[\|y_\Pi^m-y\|_{L^2}\leq C\|y_\Pi^m-P_hy\|_{\mathrm{DG}}+
\|P_hy-y\|_{L^2}\leq C(h_u^2+h^{m+1}).\]

Similarly, we have
\begin{equation*}
c\|p_\Pi^m-P_hp\|_{L^2}^2\leq a_{h}(p_\Pi^m-P_hp,p_\Pi^m
-P_hp)\leq \|g'(y_\Pi^m)-g'(y)\|_{L^2}\|p_\Pi^m-P_hp\|_{L^2},
\end{equation*}
then
\[\|p_\Pi^m-P_hp\|_{L^2}\leq C\|y_\Pi^m-y\|_{L^2}\leq C(h_u
^2+h^{m+1}).\]
Finally
\[\|p_\Pi^m-p\|_{L^2}\leq\|p_\Pi^m-P_hp\|_{L^2}+\|P_hp-p\|_{L^2}
\leq C(h_u^2+h^{m+1}).\]

We now have 
\[I_1+I_2+I_3\leq Ch_u^3+C(h_u^2+h^{m+1})\|\Pi_h^uu-u_h\|_{L^2}.\]
Therefore
$\|\Pi_h^uu-u_h\|_{L^2}\leq C(h_u^{\frac{3}{2}}+h^{m+1}).$
\end{proof}

\begin{rmk}\label{3.1}
For the case
\[U_{ad}=\{v\in L^2({\Omega}_u)\mid\gamma\leq v\leq\Gamma\ 
\mathrm{a.e.\ in}\ {\Omega}_u\},\]
where $\gamma<\Gamma$ are constants, take
\[\Omega_{u}^+=\{K_u \in \mathcal{T}_h^u\mid\gamma
<u_h^*|_{ K_u}<\Gamma\},
\quad \Omega_{u}^0=\{K_u \in \mathcal{T}_h^u
\mid u_h^* |_{K_u}=\gamma\mathrm{\ or\ }\Gamma\}.\]
The result in Lemma $\ref{lem2}$ still holds.
\end{rmk}

In order to make use of Lemma \ref{lem2}, 
we need to find a recovery operator $R_h:U\rightarrow
U$ to get a sharp estimate of $\|u^*-R_hu^*_h\|_{L^2}.$ 
More precisely, we have the inequality
\[\|u^*-R_hu^*_h\|_{L^2}\leq\|u^*-R_h(\Pi_h^u u^*)\|_{L^2}+\|
R_h\|\|\Pi_h^u u^*-u_h^*\|_{L^2}.\] 
Given the estimate for $\|\Pi_h^u u^*-u_h^*\|_{L^2}$,
 it remains to select a bounded linear operator 
$R_h$ ensuring that $\|u^*-R_h(\Pi_h^uu^*)\|_
{L^2}$ is of a higher order.

In particular, when $\Omega_u\subset\mathbb{R}^2$, the recovery
$R_hv(z)$ for the node $z=(x_0,y_0)\in\Omega_u$ can be 
calculated by a kind of Z-Z patch recovery. 
Let $K^u_1,K^u_2,\cdots, K^u_m$ be the elements in 
$\mathcal{T}_h^u$ containing $z$. Suppose the three nodes of 
$K^u_i$ are $(x_0,y_0),(x_i,y_i)$ and $(x_{i+1},y_{i+1}),i=1,2,
\ldots,m$. Define $R_hv(z):=a+bx_0+cy_0$, 
where $(a,b,c)$ is the 
solution to the following linear system.
\begin{equation*}\label{recover}
\centering
\begin{bmatrix}
\ds\sum_{i=1}^m|k_i|&\ds\sum_{i=1}^m
|k_i|\overline{x_i}&\ds\sum_{i=1}^m|k_i|\overline{y_i}\\
\ds\sum_{i=1}^m|k_i|\overline{x_i}&\ds\sum_{i=1}^m|k_i|
\overline{x_i}^2&\ds\sum_{i=1}^m|k_i|\overline{x_i}
\,\overline{y_i}\\
\ds\sum_{i=1}^m|k_i|\overline{y_i}&\ds\sum_{i=1}^m|k_i|
\overline{x_i}\,\overline{y_i}&\ds\sum_{i=1}^m|k_i|\overline
{y_i}^2
\end{bmatrix}
\begin{bmatrix}a\\ \\b\\ \\c\end{bmatrix}
\left.=\left
[\begin{array}{c}\ds{\sum_{i=1}^m}v_i
\\\ds\sum_{i=1}^mv_i\overline{x_i}\\
\ds\sum_{i=1}^mv_i\overline{y_i}
\end{array}
\right.\right].
\end{equation*}
Here ($\overline{x_i},\overline{y_i})$ 
is the barycenter of $K^u_i$. $|k_i|=m(K^u_i),v_i=\int_{K^u_i}v
\mathrm{d}x.$

Consequently, we can define
the operator $R_h: U\rightarrow \mathbb{P}_1(\mathcal{T}_h^u)$  
(the piecewise linear continuous function space on $\mathcal{T}_h^u$) 
without confusion. From \cite{bramble1977zzpatch}, we know that $R_h$ is bounded and 
if $u^*\in W^{1,\infty}(\Omega_u),u^*|_
{\Omega_u^+}\in H^2(\Omega_u^+),$ then there exists a constant 
$C=C(u^*)$ such that 
\[\|u^*-R_h(\Pi_h^u u^*)\|_{L^2}=\|u^*-R_hu^*\|_{L^2}\leq C
h_u^2.\]
Combine these results with Lemma $\ref{lem2},$ we have 
\begin{thm}\label{thm3.2}
Let $(y,u,p)$ satisfy $\eqref{eq27},$ $(y_h^{m},u_h^
{},p_h^{m})$ satisfy $\eqref{eq10}$. If Assumptions 
\ref{ass2},\ref{ass1},\ref{ass4},\ref{ass5} 
hold, $u,j^\prime(u)+B^*p \in W^{1,\infty}(\Omega_u),
y,p\in H^{m+1}(\Omega),$ 
the recovery 
operator $R_h$ is defined above, then there exists a constant $C=C(y,u,p)$ such that 
\[\|u-R_hu_h\|_{L^2}\leq C(h_u^{\frac{3}{2}}+h^{m+1}).\]
\end{thm}

\section{A Posteriori Estimates}\label{sec4}
\subsection{A General Result}\label{sec4.1}
This section conducts an a posteriori error analysis for the 
problem \eqref{eq10}. Analogous to the proof of Theorem 
\ref{thm2}, we will also construct two intermediate 
states. This recasts the error estimate for \eqref{eq10} 
into an analysis of the DG method's error for the underlying 
elliptic equation, thereby permitting the application of the 
established results from Theorem \ref{thm4}.
\begin{thm}\label{thm3}
Let $(y,u,p)$ satisfy $\eqref{eq27},$ $(y_h^{m},u_h
,p_h^{m})$ satisfy $\eqref{eq10}$. If Assumptions 
\ref{ass2},\ref{ass1},\ref{ass3} 
hold, and $y,p\in H^{2}(\Omega),$ then there exists a 
constant $C$ such that
\begin{equation*}
\begin{aligned}
&\|u-u_h\|_{L^2}+\| y-y_h^{m}\|_{L^2}+
\| p-p_h^{m}\|_{L^2}\\
\leq& C(\eta_0+\eta_1(y_h^{m})+\eta_2(y_h^{m})+\eta_3(y_h^
{m})+
\eta_1(p_h^{m})+\eta_2(p_h^{m})+\eta_3(p_h^{m})),
\end{aligned}
\end{equation*}
where 
\begin{equation*}
\begin{aligned}
&\eta_0^2=\|(I-\Pi_h^u)(j^\p(u_h)+B^*p_h^{m})\|_{L^2}^2,\\
&\eta_1^2(y_h^{m})=\sum_{K\in\mathcal{T}^h}\widetilde{h_K}^4\|
f+Bu_h+\nabla\cdot (A\nabla y_{h}^{m})\|_{L^2(K)}^{2},\\
&\eta_1^2(p_{h}^{m})=\sum_{K\in\mathcal{T}^h}\widetilde{h_K}^4\|
g^\p(y_h^{m})+\nabla\cdot (A\nabla p_{h}^{m})\|_{L^2(K)}^{2},
\end{aligned}
\end{equation*}
$\eta_2,\eta_3$ are defined by $\eqref{eq26}$.
\end{thm}

\begin{proof}
Construct the intermediate states as follows.
\begin{equation*}
\begin{aligned}
&-\nabla\cdot(A\nabla Sy)=f+Bu_h,\quad Sy|_{\partial
\Omega}=\phi,\\
&-\nabla\cdot(A\nabla Sp)=g^\prime(y_h^m),\quad Sp|_
{\partial\Omega}=0,
\end{aligned}
\end{equation*}
which means that
\begin{equation}\label{eq21}
\begin{aligned}
&a_{h}(Sy,w_h)=l_{h}[f+Bu_h,
\phi](w_h),\quad\forall w_h\in V_h,\\
&a_{h}(q_h,Sp)=l_{h}[g^\prime(y_h^m),
0](q_h),\quad\forall q_h\in V_h.
\end{aligned}
\end{equation}
From $\eqref{eq27}$ and $\eqref{eq21}$, we have
\begin{equation}\label{eq22}
\begin{aligned}
&a_{h}(y-Sy,w_h)=(B(u-u_h),w_h),\quad\forall w_h\in V_h,\\
&a_{h}(q_h,p-Sp)=
(g^\p(y)-g^\p(y_h^m),q_h),\quad\forall q_h\in V_h.
\end{aligned}
\end{equation}
Take $w_h =p-Sp,q_h =y-Sy$ in $\eqref{eq22}$ 
and subtract the resulting equations to find
$$( B^*(p-Sp), u-u_h)=(g^\p(y)-g^\p(y_h^m),y-Sy).$$
This implies
\begin{equation}\label{eq23}
( B^*(p-Sp), u-u_h)\geq\beta\|y-Sy\|_{L^2}+(g^\p(Sy)-g^\p(y_
h^m),y-Sy).
\end{equation}

In addition, by $\eqref{eq16}$ in the proof of Theorem $\ref{thm2}$, 
we know that
\begin{equation}\label{eq24}
( B^*(p_h^m - p), u - u_h) \geq \alpha \|u - u_h\|_{L^2}^2 + 
( j^\prime(u_h)+B^*p_h^m, u - v_h) ,\quad\forall v_h\in U_{adh}.
\end{equation}
Take $v_h=\Pi_h^u u.$

Combine $\eqref{eq23}$ and $\eqref{eq24}$, we have
\begin{equation*}
\begin{aligned}
&\alpha\|u-u_h\|_{L^2}^{2}+\beta\|y-Sy\|_{L^2}^{2}\\
\leq&-(j^\p(u_h)+B^*p_h^m,u-\Pi_h^uu)-(B^*(Sp-
p_h^m),u-u_{h}) +(g^\p(y_h^m)-g^\p(Sy),y-Sy) \\
  =&-((I-\Pi_h^u)(j^\p(u_h)+B^*p_h^m),u-\Pi_h^uu)\\
& -(B^*(Sp-p_h^m),u-u_{h}) +(g^\p(y_h^m)-g^\p(Sy),y-Sy)
\end{aligned}
\end{equation*}
\begin{equation*}
\begin{aligned}
  =&-((I-\Pi_h^u)(j^\p(u_h)+B^*p_h^m),u-u_h)\\
& -(B^*(Sp-p_h^m),u-u_{h}) +(g^\p(y_h^m)-g^\p(Sy),y-Sy) \\
\leq& C(\eta_0^2+\|B^*(Sp-p_h^m)\|_{L^2}^2+\|Sy-y_h^m\|
_{L^2}^{2})+\frac{\alpha}{2}\|u-u_{h}\|_{L^2}^{2}+\frac
{\beta}{2}\|y-Sy\|_{L^2}^{2}.\quad\quad\quad\,\\
\end{aligned}
\end{equation*}
So we have
\begin{equation*}
\|u-u_h\|_{L^2}+\|y-Sy\|_{L^2}
\leq C(\eta_0+\|Sp-p_h^m\|_{L^2}+\|Sy-y_h^m\|_{L^2}).
\end{equation*}

We will give upper bounds for $\| y-y_h^m
\|_{L^2}$ and $\|p-p_h^m\|
_{L^2}.$
We use the stability of $a_{h}$ on $(H^2(\Omega)\cap 
H_0^1(\Omega),\|\cdot\|_{L^2})$. Let $w_h=y-Sy$ 
and $q_h=p-Sp$ in $\eqref{eq22},$ we have 
\[\|y-Sy\|_{L^2}\leq C\|u-u_h\|_{L^2},\quad
\|p-Sp\|_{L^2}\leq C\|y-y_h^m\|_{L^2}.\] 
Thus
\begin{equation*}
\begin{aligned}
&\| p-p_h^m\|_{L^2}\leq 
C(\| Sp-p_h^m\|_{L^2}+\|y-y_h^m\|_{L^2}),\\
&\| y-y_h^m\|_{L^2}\leq 
C(\| Sy-y_h^m\|_{L^2}+\|u-u_h\|_{L^2}).
\end{aligned}
\end{equation*}
Combine the three inequalities above, we have
\begin{equation*}
\begin{aligned}
\|u-u_{h}\|_{L^2}+\| y-y_h^m\|_{L^2}
+\| p-p_h^m\|_{L^2} 
\leq
C(\eta_0+\|Sp-p_h^m\|_{L^2}+\|Sy-y_h^m\|_{L^2}).\\
\end{aligned}
\end{equation*}
Moreover, by Theorem $\eqref{thm4}$, we have
\begin{equation*}
\begin{aligned}
&\|Sp-p_h^m\|_{L^2}\leq C(\eta_1(p_h^m)+\eta_2(p_h^m)
+\eta_3(p_h^m)),\\
&\|Sy-y_h^m\|_{L^2}\leq C(\eta_1(y_h^m)+\eta_2(y_h^m)
+\eta_3(y_h^m)),
\end{aligned}
\end{equation*}
which leads to the conclusion.
\end{proof}

\subsection{Sharp A Posteriori Estimate.}\label{sec4.2}
We concentrate on the case where
\[U_{ad}=\{v\in L^2({\Omega}_u)\mid v\geq\gamma\ \mathrm
{a.e.\ in}\ {\Omega}_u\},\]
where $\gamma$ is a constant. We divide ${\Omega}_u$
into two disjoint subsets.
\begin{equation*}
\begin{aligned}
\widehat{\Omega}_{u}^+ =\{K_u \in \mathcal{T}_h^u\mid 
u_h^*|_{ K_u}>\gamma\},\quad
\widehat{\Omega}_{u}^0=\{K_u \in \mathcal{T}_h^u\mid 
u_h^*|_{ K_u}=\gamma\}.
\end{aligned}
\end{equation*}
Consider a subset of $\widehat{\Omega}_{u}^0$ 
\[\widehat{\Omega}_{u1}^0=\{x\in\widehat{\Omega}_{u}^0\mid j'(u_h^*)+B^*p_h^
{m*}\geq0\}.\]

\begin{thm}\label{thm5}
Let $(y,u,p)$ satisfy $\eqref{eq27},$ $(y_h^{m},u_h
,p_h^{m})$ satisfy $\eqref{eq10}$. If Assumptions 
\ref{ass2},\ref{ass1} hold, $y,p\in H^{2}({\Omega}),$ 
then there exists a constant $C$ such that
\begin{equation*}
\begin{aligned}
&\eta_*+\|u-u_h\|_{L^2}+\| y-y_h^{m}\|_{L^2}+
\| p-p_h^{m}\|_{L^2}\\
\leq& C(\widetilde{\eta_0}+\eta_1(y_h^{m})+\eta_2(y_h^{m})
+\eta_3(y_h
^{m})+\eta_1(p_h^{m})+\eta_2(p_h^{m})+\eta_3(p_h^{m})),
\end{aligned}
\end{equation*}
where 
\begin{equation*}
\begin{aligned}
\eta_*\, =\|(I-\Pi_h^u)(j^\p(u)+B^*p)\|_{L^2
(\widehat{\Omega}_{u}^{+})},\quad\widetilde{\eta_0}\,=\|j^\p(u_h)+
B^*p_h^{m}\|_{L^2((\widehat{\Omega}_{u1}^{0})^c)}.
\end{aligned}
\end{equation*}
$\eta_1,\eta_2,\eta_3$ are defined in Theorem $\ref{thm3}$. 
Here $(\widehat{\Omega}_{u1}^{0})^c=\Omega_u\setminus
\widehat{\Omega}_{u1}^{0}.$
\end{thm}

\begin{proof}
For $\forall K\in
\mathcal{T}_h^u,$ take $v_h=\gamma\chi_{K}+u_h\chi_{{\Omega}_u
\setminus K}\in 
U_{adh}.$ Here $\chi_K$ is the indicator function of $K$. We know by 
$\eqref{eq10}$ that 
\[(\Pi_h^u(j'(u_h)+B^*p_h^m),v_h-u_h)=(j'(u_h)+B^*p_h^m,
v_h-u_h)\geq0,\] 
which implies $\Pi_h^u(j'(u_h)+B^*p_h^m)=0$ on $\widehat{\Omega}_{u}^+.$
\begin{equation*}
\begin{aligned}
\eta_*=&\|j^\p(u)+B^*p-\Pi_h^u(j^\p(u)+B^*p)\|_{L^2
(\widehat{\Omega}_{u}^{+})}\\
\leq&\|j^\p(u)+B^*p-j'(u_h)-B^*p_h^m\|_{L^2
(\widehat{\Omega}_{u}^{+})}+
\|(I-\Pi_h^u)(j^\p(u_h)+B^*p_h^{m})\|_{L^2
(\widehat{\Omega}_{u}^{+})}\\
&+\|\Pi_h^u(j^\p(u_h)+B^*p_h^{m}-j^\p(u)-B^*p)\|
_{L^2(\widehat{\Omega}_{u}^{+})}\\
\leq&C(\|j^\p(u)+B^*p-j'(u_h)-B^*p_h^m\|_{L^2
(\widehat{\Omega}_{u}^{+})}+\|j^\p(u_h)+B^*p_h^{m}\|_{L^2
(\widehat{\Omega}_{u}^{+})})\\
\leq&C(\|u-u_h\|_{L^2(\widehat{\Omega}_{u}^{+})}+\|p-p_h^m\|
_{L^2(\widehat{\Omega}_{u}^{+})}+\|j'(u_h)+B^*p_h^m\|_{L^2(\widehat{\Omega}_
{u}^{+})})\\
\leq&C(\|u-u_h\|_{L^2}+\|p-p_h^m\|
_{L^2}+\widetilde{\eta_0}).
\end{aligned}
\end{equation*} 

Notice that $j^\p(u_h)+B^*p_h^m\geq 0$ on $\widehat{\Omega}_{u1}^{0}$, we have
\begin{equation*}
\begin{aligned}
&(j^\p(u_h)+B^*p_h^m,u_h-u)\\
=&(j^\p(u_h)+B^*p_h^m,u_h-u)_{L^2(\widehat{\Omega}_{u1}^{0})}+
(j^\p(u_h)+B^*p_h^m,u_h-u)_{L^2((\widehat{\Omega}_{u1}^{0})^c)}\\
\leq&0+C\widetilde{\eta_0}^2+\frac{\alpha}{2}\|u_h-u\|_{L^2
((\widehat{\Omega}_{u1}^{0})^c)}^2\leq C\widetilde{\eta_0}^2+\frac
{\alpha}{2}\|u_h-u\|_{L^2}^2.
\end{aligned}
\end{equation*} 

Follow the proof of Theorem $\ref{thm3}$ but taking 
$v_h=u_h$ in $\eqref{eq24}$, 
\begin{equation*}
\begin{aligned}
&\alpha\|u-u_h\|_{L^2}^{2}+\beta\|y-Sy\|_{L^2}^{2}\\
\leq&(j^\p(u_h)+B^*p_h^m,u_h-u)-(B^*(Sp-
p_h^m),u-u_{h}) +(g^\p(y_h^m)-g^\p(Sy),y-Sy) \\
\leq& C(\widetilde{\eta_0}^2
+\|Sp-p_h^m\|_{L^2}^2
+\|Sy-y_h^m\|_{L^2}^{2})+\frac{3\alpha}{4}
\|u-u_{h}\|_{L^2}^{2}+\frac{\beta}{2}\|y-Sy\|_{L^2}^{2}.
\end{aligned}
\end{equation*}
So we have
\begin{equation*}
\|u-u_h\|_{L^2}+\|y-Sy\|_{L^2}
\leq C(\widetilde{\eta_0}+\|Sp-p_h^m\|_{L^2}+\|
Sy-y_h^m\|_{L^2}).
\end{equation*}
Using the estimate of $\eta_*$ as above, we can finish the proof by the 
same procedure in Theorem $\ref{thm3}$.
\end{proof}

\begin{rmk}\label{rmk3}
For the case where
\[U_{ad}=\{v\in L^2({\Omega}_u)\mid\gamma\leq v\leq\Gamma\ 
\mathrm{a.e.\ in}\ {\Omega}_u\},\]
where $\gamma<\Gamma$ are constants, take
\[\widehat{\Omega}_{u}^+=\{K_u \in \mathcal{T}_h^u\mid
\gamma<u_h^*|_{ K_u}<\Gamma\},
\quad\widehat{\Omega}_{u}^0=\{K_u \in \mathcal{T}_h^u\mid
u_h^* |_{K_u}=\gamma\mathrm{\ or\ }\Gamma\},\]
and 
\begin{equation*}
\begin{aligned}
\widehat{\Omega}_{u1}^0=&\{x\in\widehat{\Omega}_u^0\mid 
u_h^*(x)=\gamma,j'(u_h^
*)+B^*p_h^{m*}\geq0\}
\scalebox{1.3}{$\cup$}\\
&\{x\in\widehat{\Omega}_u^0\mid u_h^*(x)=\Gamma,j'
(u_h^*)+B^*p_h^{m*}\leq0\}.
\end{aligned}
\end{equation*}
The inequality in Theorem $\ref{thm5}$ still holds.
\end{rmk}
\section{Numerical Results}\label{sec5}
In thus section, 
we provide some numerical examples to validate 
the derivation in the previous sections. In Section \ref{sec5.1}, 
we verify the convergence order of a priori estimate proved 
in Theorem \ref{thm2} and Theorem \ref{thm3.2}. In 
Section $\ref{sec5.2}$, we study the case that the control 
variable has 
sufficiently smooth data. We build another approximation to 
the optimality condition $\eqref{eq27}$
without using $\mathcal{T}_h^u$, which 
means that we do not discretize the control variable. 
In Section \ref{sec5.3}, 
we demonstrate the efficiency of  a posteriori estimate proved 
in Theorem \ref{thm3}. 



\subsection{Examples for Convergence Order.}\label{sec5.1}
We will give some examples to verify the convergence order 
proved in Theorem $\ref{thm2}$ and Theorem 
$\ref{thm3.2}.$ 
We always take the penalty parameter
$\mu=3m^2$ in the examples below. Notice that the 
right-hand side of the estimates in theses two theorems has 
the form $C(h_u+h^m),C(h_u+h^{m+1})$ or $C(h_u^{\frac{3}{2}}+h^{m+1})$. If we want 
to see the convergence order clearly, we must take $h_u\propto
h^m$ or $h^{m+1}$ to meet the order of $h$. Then we can find 
the values of a priori error have convergence order $m$ or $m+1$, 
as shown below.

\noindent\textbf{Example 1.} Take $\Omega_u=\Omega=(0,1)^2,
B=\mathrm{Id}.$
\begin{equation*}
\begin{aligned}
&\min_{u\in L^2(\Omega_u)}\frac{1}{2}\|y-y_d\|_
{L^2}^2+\frac{1}{2}\|u-u_d\|_{L^2}^2,
&\\\mathrm{s.t.}\ &-\Delta y=f+u,\quad y|_{\partial\Omega}=y_d|_
{\partial\Omega},\\
& u\geq 0,\quad \mathrm{a.e.\ in}\ \Omega_u.
\end{aligned}
\end{equation*}
where
\begin{equation*}
\begin{aligned}
&y_d=0,\quad u_d=1-\mathrm{sin}\frac{\pi x_1}{2}-\mathrm{sin}
\frac{\pi x_2}{2},\\
&p=\mathrm{sin} (\pi x_1) \mathrm{sin} (\pi x_2),\quad f=4\pi^4
p-\mathrm{max}\{u_d-p,0\}.
\end{aligned}
\end{equation*}
The exact solution
\begin{equation*}
y^*=2\pi^2p+y_d,\quad u^*=\mathrm{max}\{u_d-p,0\},\quad 
p^*=p.
\end{equation*}
\begin{figure}[H] \label{fig_ex11_L2err}
\centering
\includegraphics[width=0.4\textwidth]{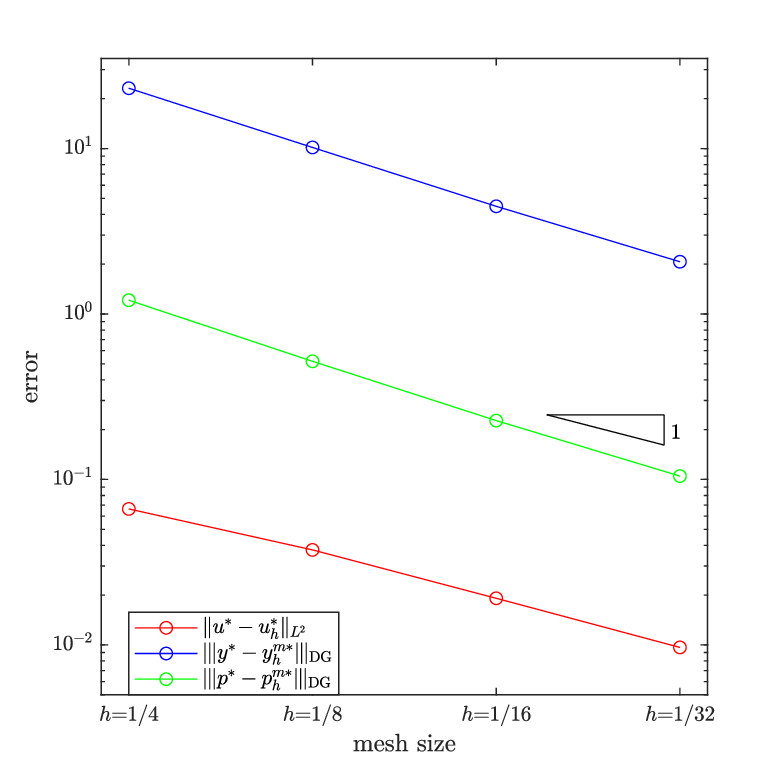}
\includegraphics[width=0.4\textwidth]{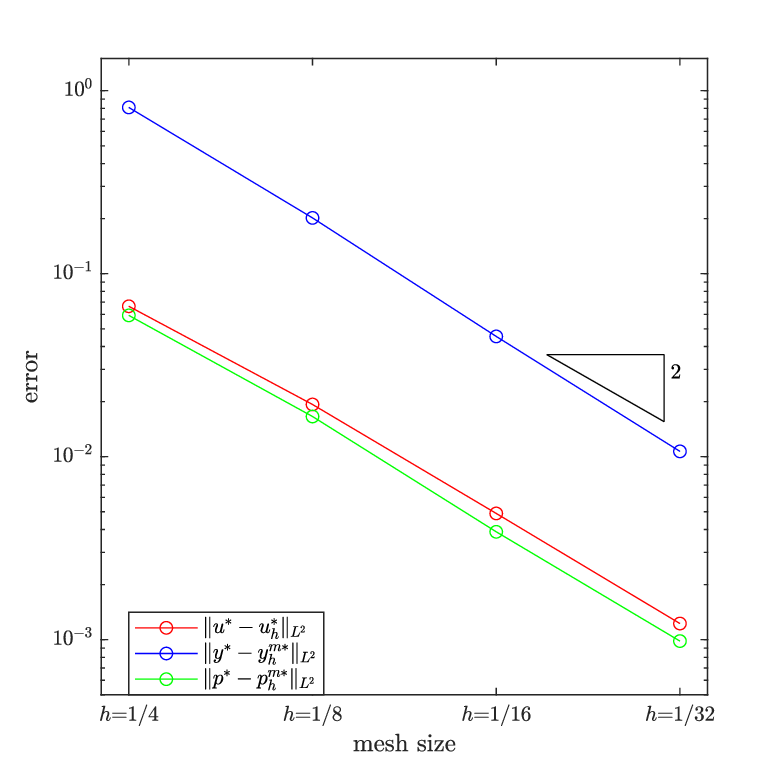}
\caption{Error of Example 1 for $m=1.$ Take $h_u=h$(left) and $h_u=4h^2$(right).}
\end{figure}\vspace{-0pt}
\begin{figure}[H] \label{fig_ex12_L2err}
\centering
\includegraphics[width=0.4\textwidth]{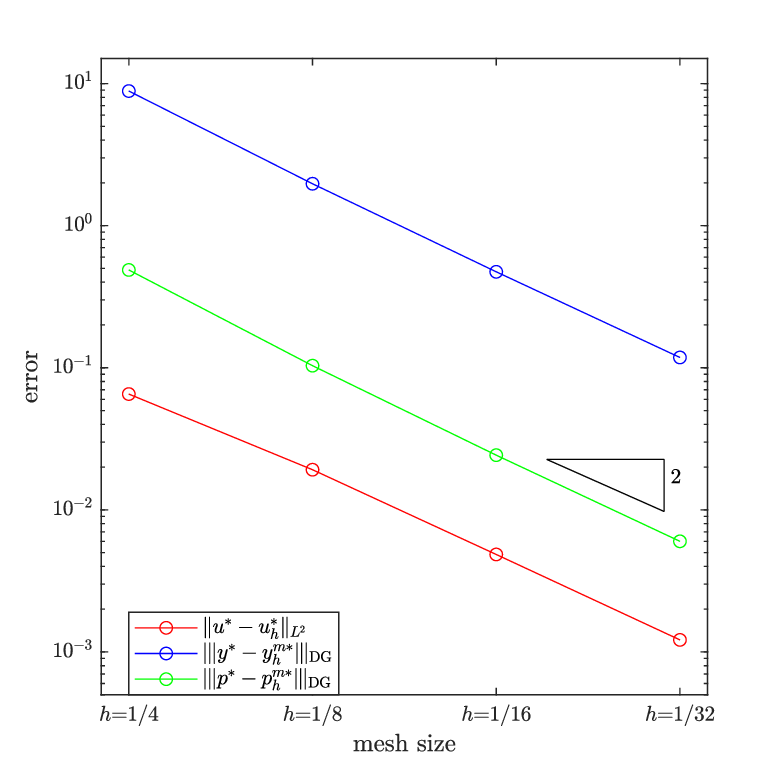}
\includegraphics[width=0.4\textwidth]{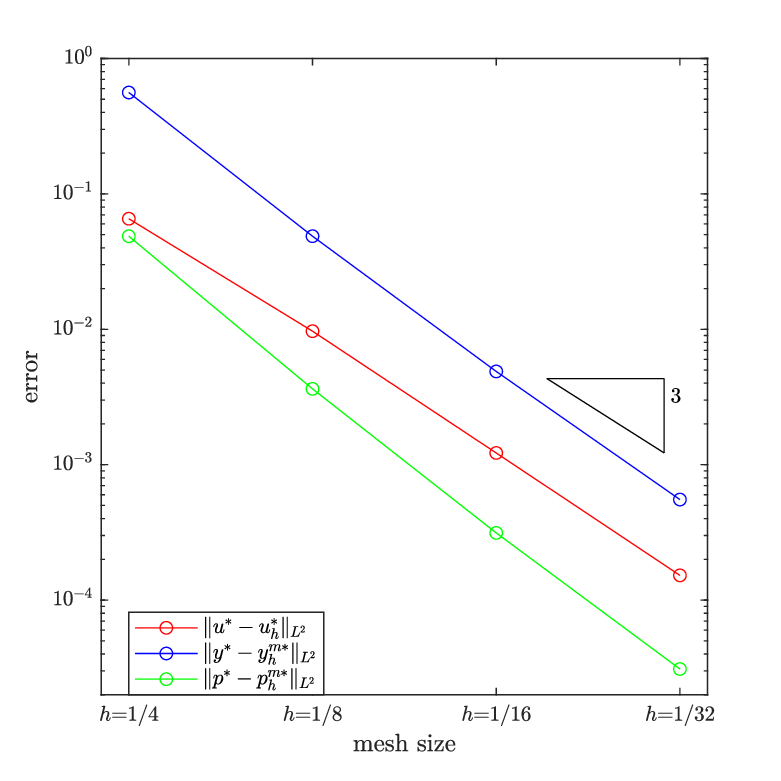}
\caption{Error of Example 1 for $m=2.$ Take $h_u=4h^2$(left) and $h_u=16h^3$(right).}
\end{figure}
\begin{figure}[H]
\centering
\captionsetup{labelsep=space} 
\begin{minipage}{0.48\textwidth}
  \centering
  \includegraphics[width=0.83\textwidth]{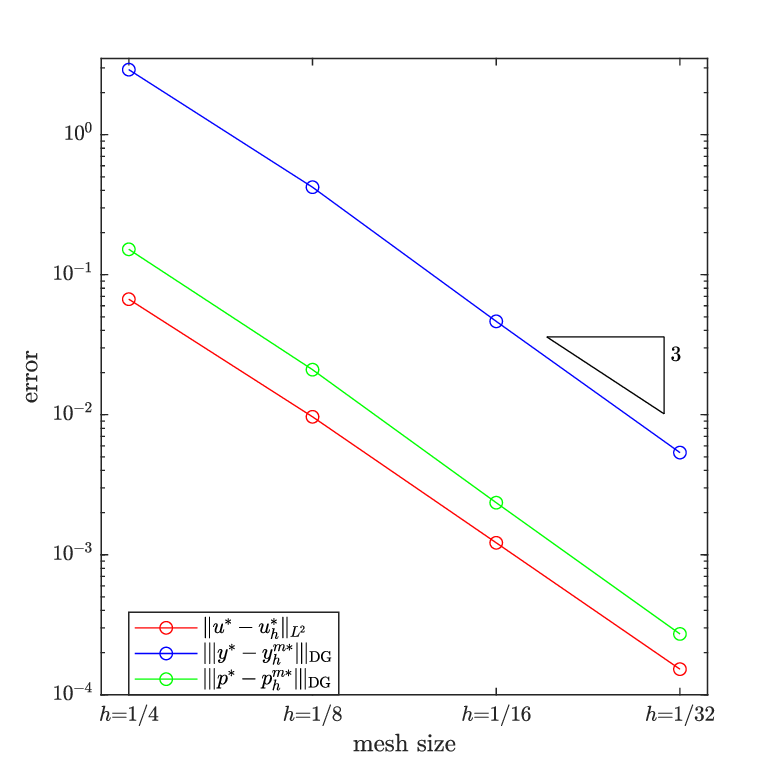}
  \caption{Error of Example 1 for $m=3$. Take $h_u=16h^3$.}
  \label{fig_ex13_L2err}
\end{minipage}\hspace{-37.5pt}
\begin{minipage}{0.48\textwidth}
  \centering
  \includegraphics[width=0.83\textwidth]{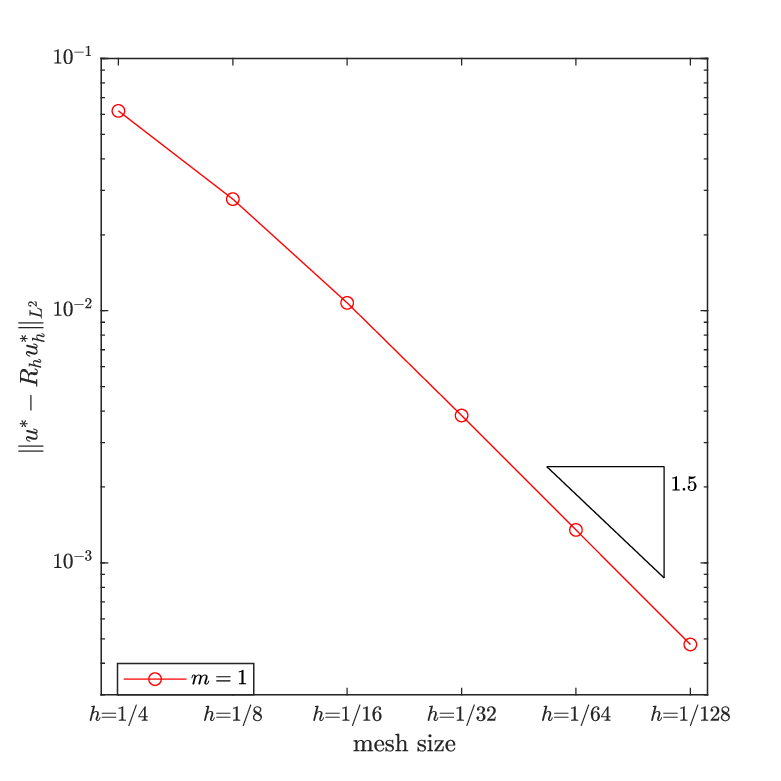}
  \caption{Recovery error of Example 1. 
Take $h_u=h.$}
  \label{fig_ex1sup_L2err}
\end{minipage}
\end{figure}\vspace{-0pt}

\noindent\textbf{Example 2.} Take $\Omega_u=\Omega=(0,1)^2,
B=\ds\frac{1}{2}\mathrm{Id}.$
\begin{equation*}
\begin{aligned}
&\min_{u\in L^2(\Omega_u)}\frac{1}{2}\|y-y_d\|_{L^2}^2+
\frac{1}{2}\|u-u_d\|_{L^2}^2+\frac{4}{3}\|u-u_d\|_{L^3}^3,
&\\\mathrm{s.t.}\ &-\Delta y=f+\frac{1}{2}u,\quad y|_{\partial\Omega}=y_d|_
{\partial\Omega},\\
& 0\leq u\leq 1,\quad \mathrm{a.e.\ in}\ \Omega_u.
\end{aligned}
\end{equation*}
where
\begin{equation*}
\begin{aligned}
&u_d=\frac{3}{2}-x_1- x_2,\quad p=\mathrm{sin} (\pi x_1) \mathrm
{sin} (\pi x_2),\\
&y_d=0,\quad f=4
\pi^4p-\frac{1}{2}Pr_{[0,1]}(u_d-s(p)).
\end{aligned}
\end{equation*}
The exact solution
\begin{equation*}
y^*=2\pi^2p+y_d,\quad u^*=Pr_{[0,1]}(u_d-s(p)),\quad 
p^*=p,
\end{equation*}
where $s(p)$ is an implicit function satisfies \[s(p)+4
s(p)^2\mathrm{sgn}(s(p))=\frac{1}{2}p,\]
and $Pr_{[0,1]}(s):=\min\{\max\{s,0\},1\}.$
\begin{figure}[H] \label{fig_ex31_L2err}
\centering
\includegraphics[width=0.4\textwidth]{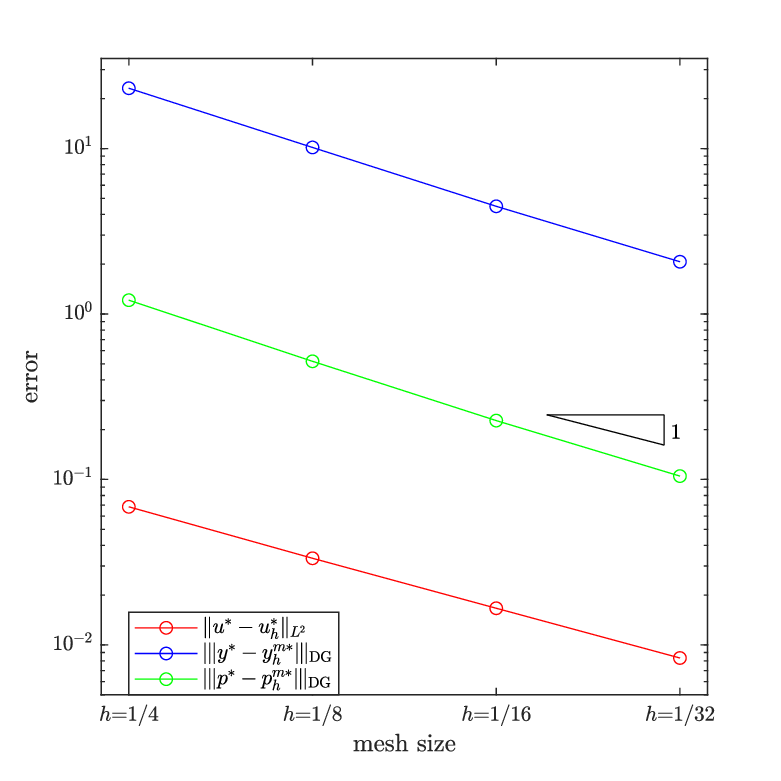}
\includegraphics[width=0.4\textwidth]{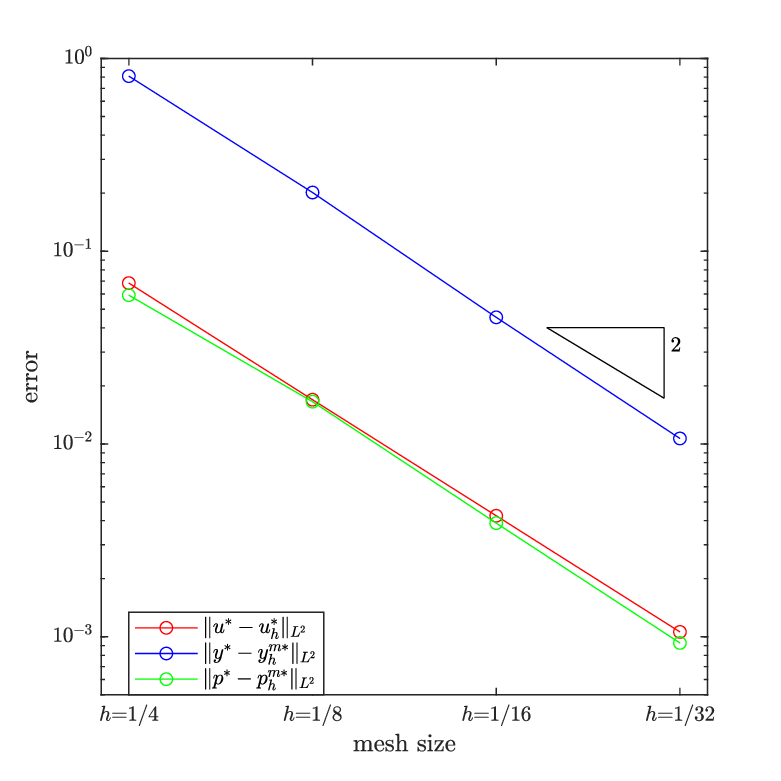}
\caption{Error of Example 2 for $m=1.$ Take $h_u=h$(left) and $h_u=4h^2$(right).}
\end{figure}\vspace{-0pt}
\begin{figure}[H] \label{fig_ex32_L2err}
\centering
\includegraphics[width=0.4\textwidth]{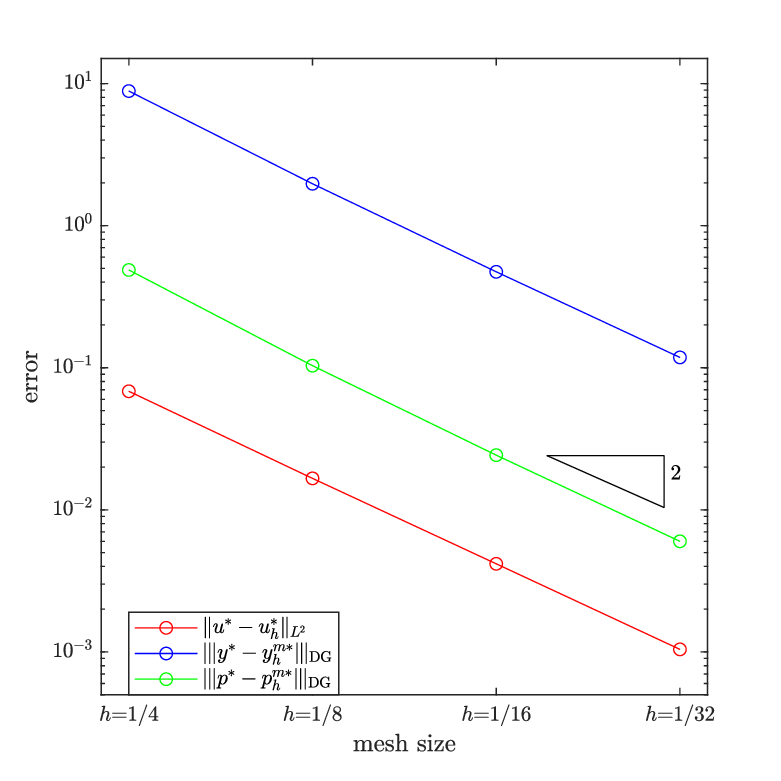}
\includegraphics[width=0.4\textwidth]{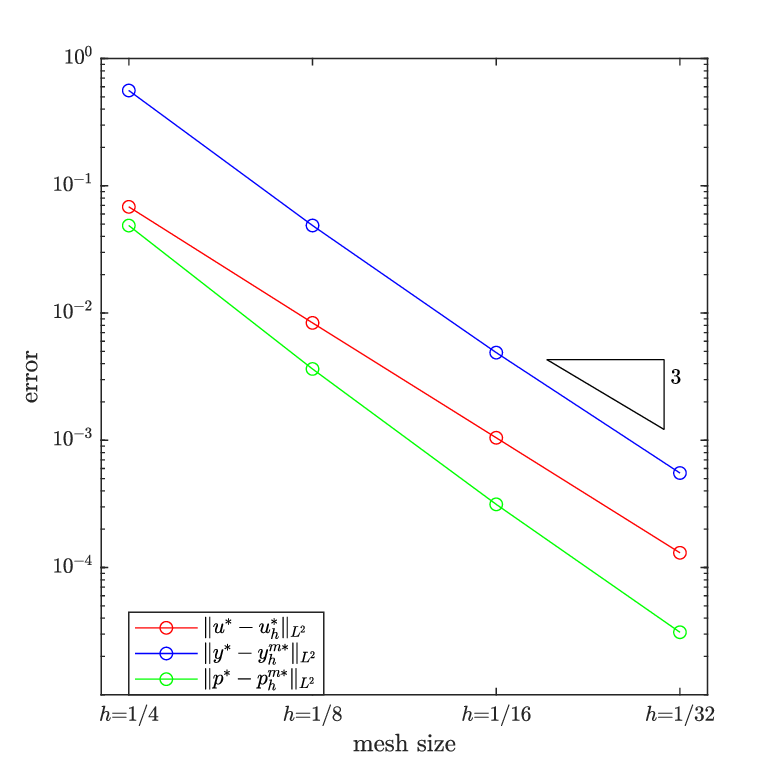}
\caption{Error of Example 2 for $m=2$. Take $h_u=4h^2$(left) and $h_u=16h^3$(right).}
\end{figure}
\begin{figure}[H]
\centering
\captionsetup{labelsep=space} 
\begin{minipage}{0.48\textwidth}
  \centering
  \includegraphics[width=0.83\textwidth]{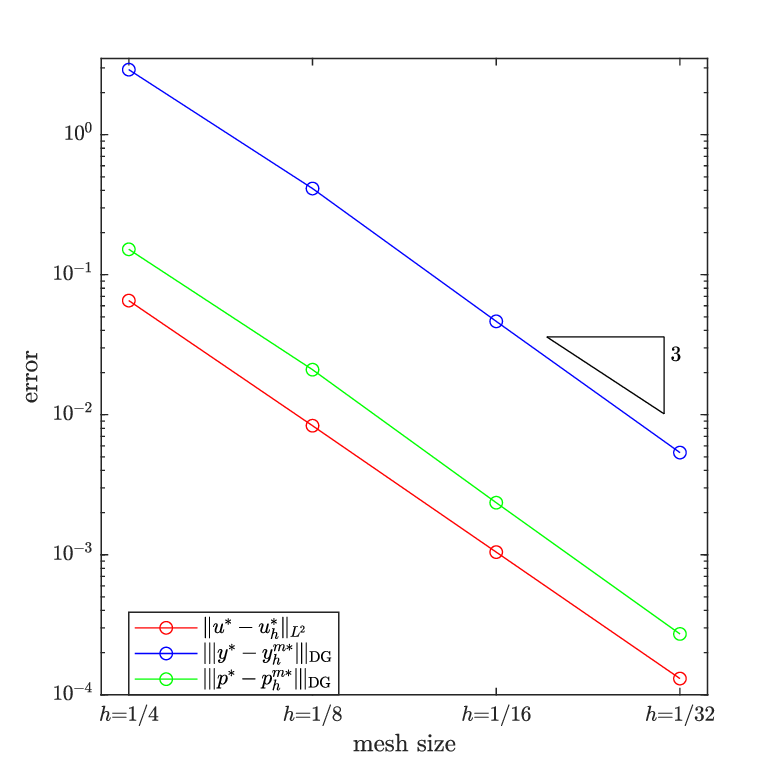}
  \caption{Error of Example 2 for $m=3$. Take $h_u=16h^3.$}
  \label{fig_ex33_L2err}
\end{minipage}\hspace{-37.5pt}
\begin{minipage}{0.48\textwidth}
  \centering
  \includegraphics[width=0.83\textwidth]{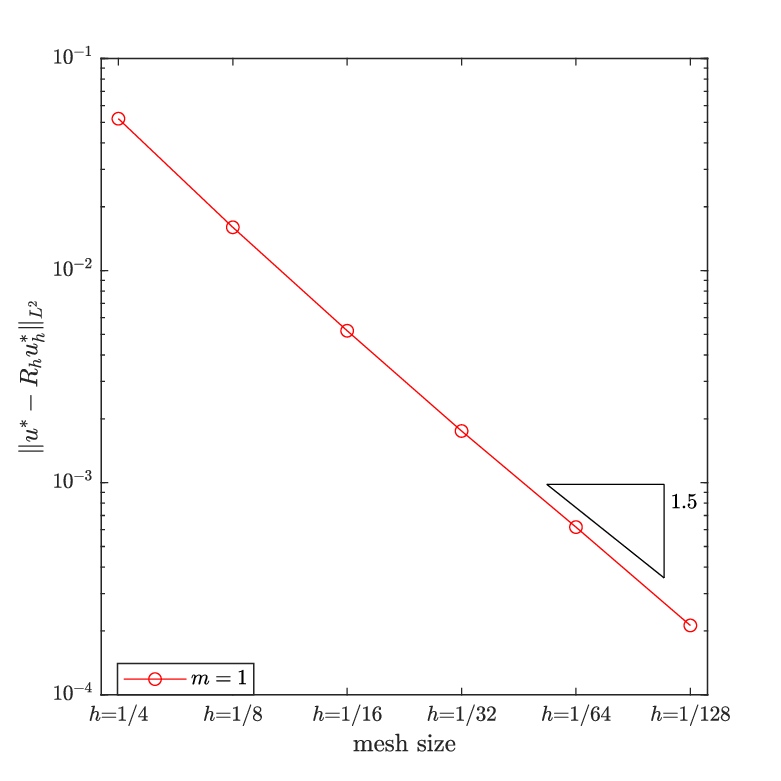}
  \caption{Recovery error of Example 2. Take $h_u=h.$}
  \label{fig_ex3sup_L2err}
\end{minipage}
\end{figure}

\subsection{Smooth Example.}\label{sec5.2}
Notice that in Theorem $\ref{thm2}$ and Theorem 
$\ref{thm3.2}$, the order 
of $h_u$ is higher compared to $h$, which leads to more 
refinement on 
$\mathcal{T}_h^u$. In order to optimize computational
 efficiency and minimize
memory cost, we can calculate $u$ without discretization 
and eliminate 
$\mathcal{T}_h^u$, i.e. we consider an approximation 
to $\eqref{eq27}$ better
than $\eqref{eq10}$ as follows.
\begin{equation}\label{eq5.3}
\left\{
\begin{aligned}
&a_{h}(\widetilde{y}_h^{m\ast},w_h^m)=l_{h}[f+B\widetilde{u}^*,\phi]
(w_h^m),\quad\forall w_h^m\in V_h^m,\\
&a_{h}(q_h^m,\widetilde{p}_h^{m\ast})=l_{h}[g^\prime(\widetilde{y}_h^
{m\ast}),0](q_h^m),\quad\forall q_h^m\in V_h^m,\\
&(j^\prime(\widetilde{u}^*)+B^*\widetilde{p}_h^{m*},v-\widetilde{u}^*)\geq 0,
\quad \forall v\in U_{ad}.
\end{aligned}
\right.
\end{equation}

We denote the solution to problem $\eqref{eq5.3}$ by 
($\widetilde{y}_h^{m\ast},\widetilde{u}^*,\widetilde{p}_h^{m*}$). 
Follow the proof of Theorem $\ref{thm2}$, we can take $v_h=u^*$ 
to derive
\[\|u^*-\widetilde{u}^*\|_{L^2}^{2}+\|\widetilde{y}_{h}^{m*}-P_{h}{y}
\|_{L^2}^{2}
\leq C(\|p^*-P_hp\|_{L^2}^2+\|y^*-P_{h}y\|_{L^2}^{2})\]
where $C$ depends on $\alpha,\beta.$ Consequently, we 
have
\[\|u^*-\widetilde{u}^*\|_{L^2}+
\interleave y^*-\widetilde{y}_h^{m*}\interleave_{\mathrm{DG}}+\interleave p^*-\widetilde{p}_h^{m*}\interleave_{\mathrm{DG}}\leq Ch^m,\]
and
\[\|u^*-\widetilde{u}^*\|_{L^2}+
\| y^*-\widetilde{y}_h^{m*}\|_{L^2}+\| p^*-\widetilde{p}_h^{m*}\|_{L^2}\leq Ch^{m+1}.\]

In each iteration, we rewrite the third equation in 
$\eqref{eq5.1}$ as 
\[\widetilde{u}_{ \mkern+1mu n \mkern-2mu+\mkern-2mu 1}
=Pr_{U_{ad}}(\widetilde{u}_{n}-
\rho_n(j'(\widetilde{u}_{n})+B^*\widetilde{p}_{hn}^m)).\] 
We can prove that the linear convergence rate still holds. 
The smoothness of $u^*$ permits high-order accurate numerical 
integration of $\widetilde{u}_n$-dependent integrands by evaluating directly at 
the quadrature points of $\mathcal{T}_h.$ Consequently, we only 
need to store the values of $\widetilde{u}_n$ at these quadrature points.
This modification serves to significantly improve 
computational efficiency and offers the advantage of reduced 
memory consumption. 

\noindent\textbf{Example 3.} Take $\Omega_u=\Omega=(0,1)^2,
B=\mathrm{Id}.$
\begin{equation*}
\begin{aligned}
&\min_{u\in L^2(\Omega_u)}\frac{1}{2}\|y-y_d\|_{L^2}^2+
\frac{1}{2}\|u-u_d\|_{L^2}^2,
&\\\mathrm{s.t.}\ &-\Delta y=f+u,\quad y|_{\partial\Omega}=y_d|_
{\partial\Omega},\quad\int_{\Omega_u}u\mathrm{d}x\geq0.
\end{aligned}
\end{equation*}
where
\begin{equation*}
\begin{aligned}
&u_d=1-2x_1-2x_2,\quad 
p=\mathrm{sin} (\pi x_1) \mathrm{sin} (\pi x_2),\\
&y_d=0,
\quad f=4\pi^4p-u_d+p+\min\{0,\overline{u_d-p}\}.
\end{aligned}
\end{equation*}
The exact solution
\begin{equation*}
y^*=2\pi^2p+y_d,\quad u^*=u_d-p-\min\{0,\overline{u_d-p}\},
\quad p^*=p,
\end{equation*}
here $\overline{u_d-p}=\ds\frac{1}{m(\Omega)}\int_{\Omega}(
u_d-p)\mathrm{d}x.$

\begin{figure}[H] \label{fig_ex41_L2err}
\centering
\includegraphics[width=0.4\textwidth]{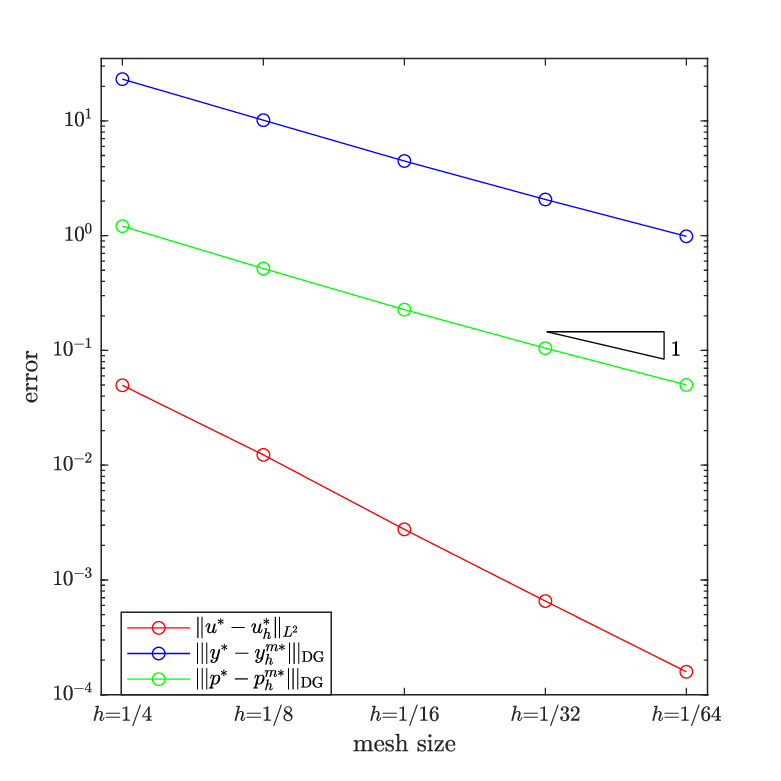}
\includegraphics[width=0.4\textwidth]{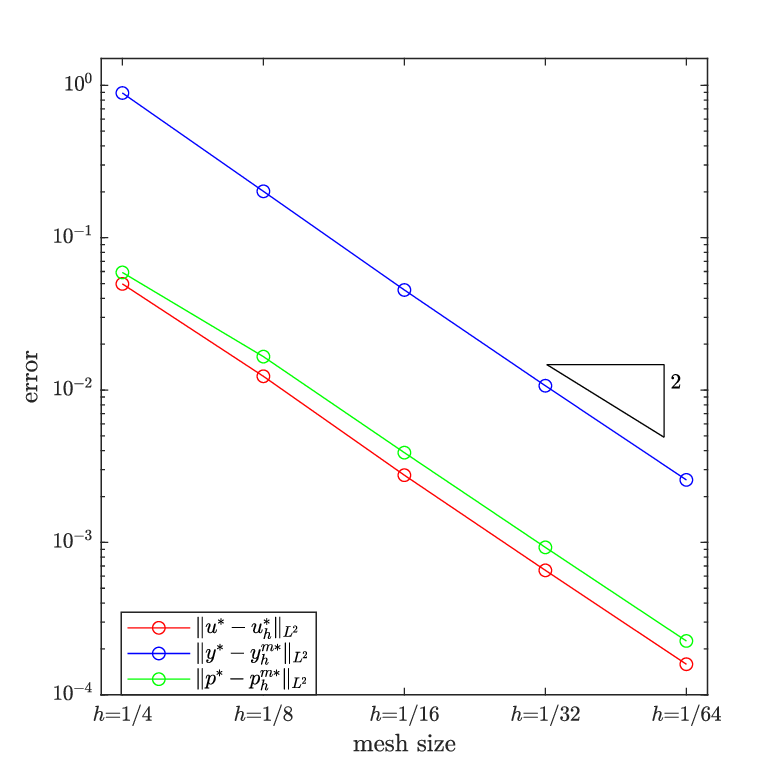}
\caption{Error of Example 3 for $m=1.$}
\end{figure}
\begin{figure}[H] \label{fig_ex42_L2err}
\centering
\includegraphics[width=0.4\textwidth]{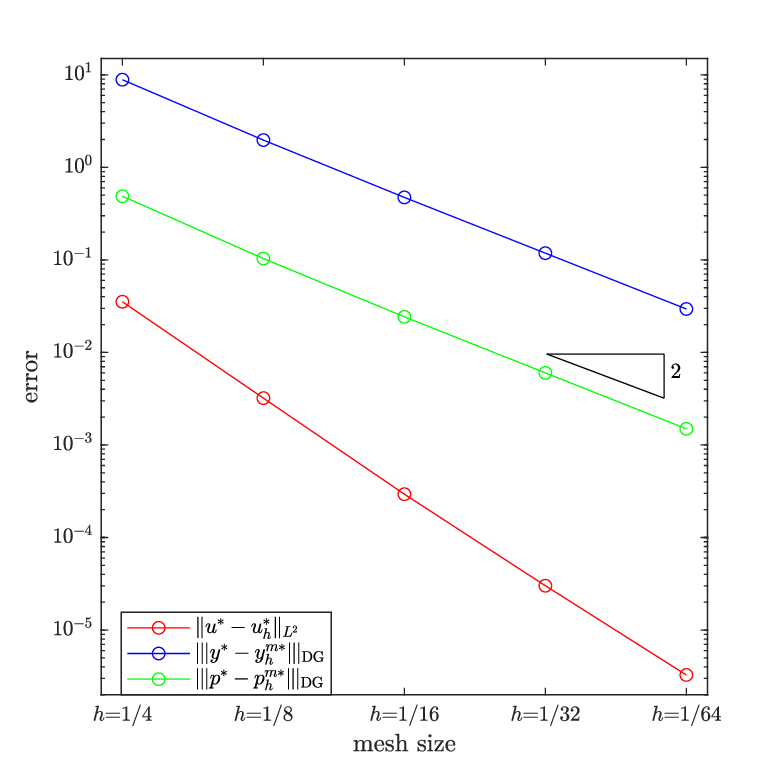}
\includegraphics[width=0.4\textwidth]{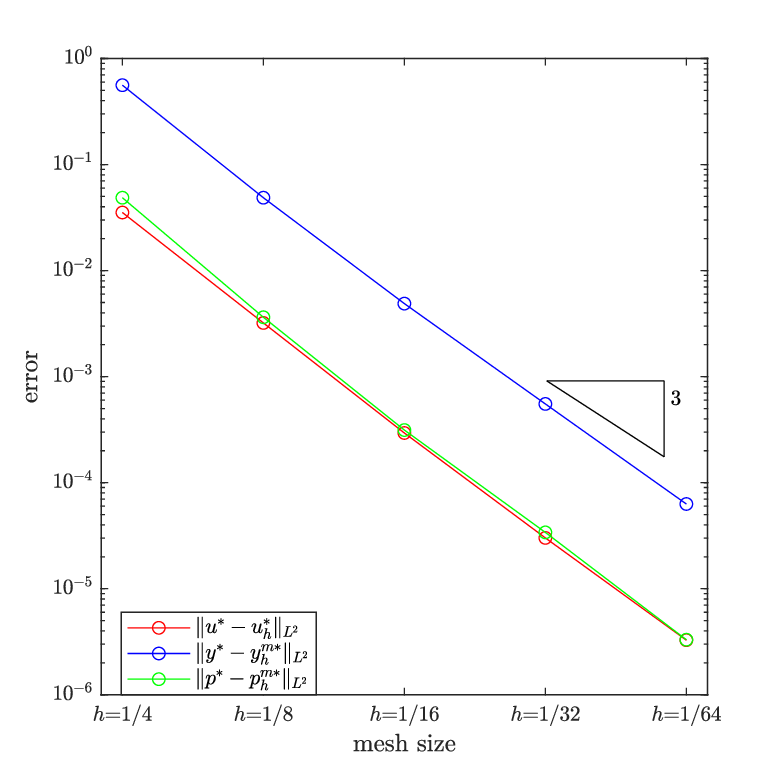}
\caption{Error of Example 3 for $m=2.$}
\end{figure}
\begin{figure}[H] \label{fig_ex43_L2err}
\centering
\includegraphics[width=0.4\textwidth]{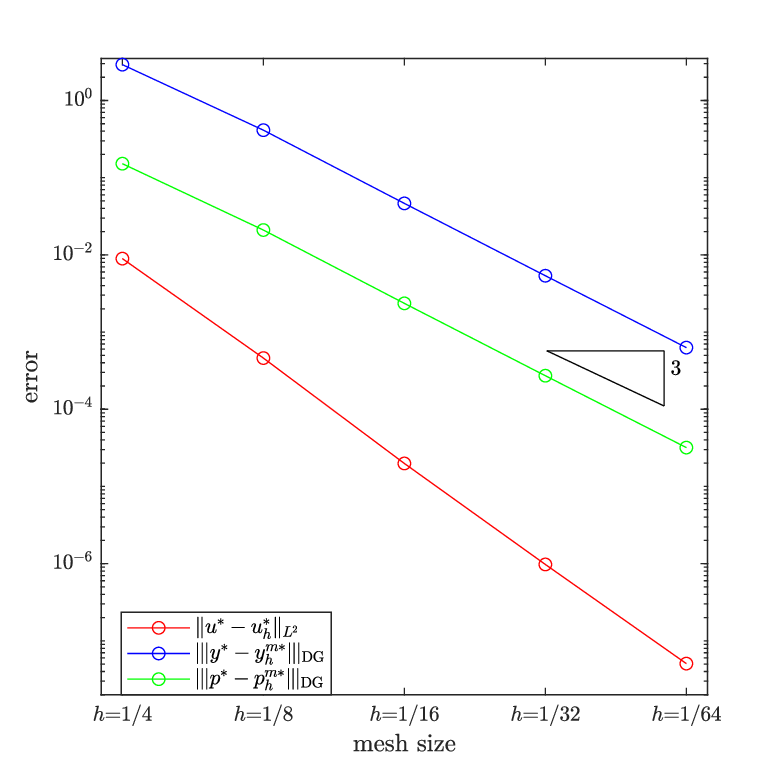}
\includegraphics[width=0.4\textwidth]{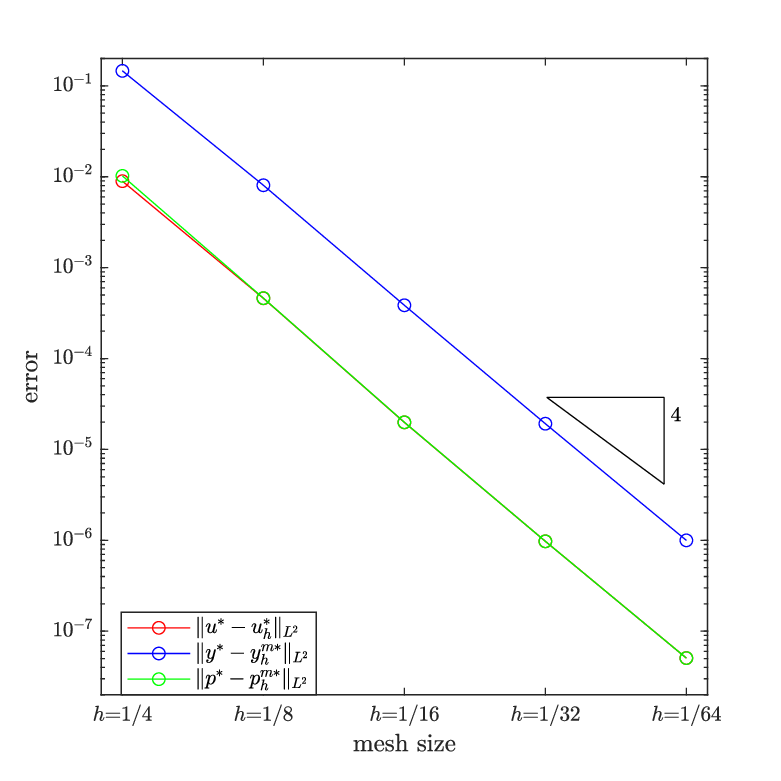}
\caption{Error of Example 3 for $m=3.$}
\end{figure}

\subsection{A Posteriori Estimator.}\label{sec5.3}
We will show the efficiency of a posteriori error estimator 
provided in Theorem $\ref{thm3}$. We will plot the ratio of a 
posteriori error indicator 
and local error for three variables $(y,u,p)$. For the control  
variable $u$, a posteriori error estimators
\[e_{K_u}^2:=\frac{\|j'(u^*_h)+B^*p_h^{m*}\|^2_{L^2(K_u)}}
{\|u^*-u_h^*\|^2_{L^2(K_u)}},\quad\forall K_u\in\mathcal{T}_h^u.\]
For the variables $y,p$ and $K\in\mathcal{T}_h$, we write 
\begin{equation*}
\begin{aligned}
&\eta_{K_y}^2:={\widetilde{h_K}^4\|
f+Bu_h^*+\nabla\cdot (A\nabla y_{h}^{m*})\|_{L^2(K)}^{2}
+\sum_{e\in \mathcal{E}_K\cap\mathcal{E}_{h}}\widetilde{h_e}\|
[y_h^{m*}]\|_{L^2(e)}^2+
\sum_{e\in \mathcal{E}_K\cap\mathcal{E}_h^I}\widetilde{h_e}^3\|
[A\nabla y_{h}^{m*}]\|_{L^2(e)}^{2}}
,\\
&\eta_{K_p}^2:={\widetilde{h_K}^4\|
g'(y_h^{m*})+\nabla\cdot (A\nabla p_{h}^{m*})\|_{L^2(K)}^{2}
+\sum_{e\in \mathcal{E}_K\cap\mathcal{E}_{h}}\widetilde{h_e}\|
[p_h^{m*}]\|_{L^2(e)}^2+
\sum_{e\in \mathcal{E}_K\cap\mathcal{E}_h^I}\widetilde{h_e}^3\|
[A\nabla p_{h}^{m*}]\|_{L^2(e)}^{2}}.
\end{aligned}
\end{equation*}
We define
\[e_{K_y}^2:=\frac{\eta_{K_y}^2}{\|y^*-y_h^{m*}\|_{L^2(K)}^2},\quad
e_{K_p}^2:=\frac{\eta_{K_p}^2}{\|p^*-p_h^{m*}\|_{L^2(K)}^2},\quad
\forall K\in\mathcal{T}_h.\]
In addition, these expressions are defined as 0 whenever their 
 denominators are 0.

\noindent\textbf{Example 4.} 
We will present some results of 
$e_{K_y},e_{K_p}$ and $e_{K_u}$ for the problem in Example 1. 
We can see that the maximum and minimum values 
of $e_{K_y}$ 
and $e_{K_p}$ show no significant variation 
with the refinement of $\mathcal{T}_h$, 
illustrating the efficiency $\eta_{K_y}$ and $\eta_{K_p}$.
\begin{table}[H]
\centering
\renewcommand\arraystretch{1.0}
\begin{tabular}{l|l|l|l|l|l|l}
\hline
$h$&1/4&1/8&1/16&1/32&1/64&1/128\\
\hline
$\max_K e_{K_y}$&9.656&18.10&20.88&24.28&25.99&26.63\\
\hline
$\min_K e_{K_y}$&2.375&2.312&2.112&1.965&1.807&1.620\\
\hline
$\max_K e_{K_p}$&9.837&18.03&21.49&23.92&24.56&23.98\\
\hline
$\min_K e_{K_p}$&1.406&1.294&1.284&1.333&1.222&1.093\\
\hline
\end{tabular}
\caption{$e_{K_y}$ and $e_{K_p}$ of Example 4. Take $m=1,
h_u=h.$}
\end{table}
Next we show the image of $e_{K_u}$ for $m=1$. Here we plot 
$\min\{e_{K_u}, 2\}$ instead of $e_{K_u}$, for the reason that 
the maximum value of $e_{K_u}$ is too large and always
exceeds 1000. We can see that the values of 
$e_{K_u}$ in the region $\{u^*>0\}$ become near to 1 with the 
refinement of $\mathcal{T}_h,$ indicating that $\eta_{K_u}$ is 
efficient. The borderline between two 
regions $\{e_{K_u}>0\}$ and $\{e_{K_u}=0\}$ 
tends to the curve $\{u_d-p^*=0\}.$

\begin{figure}[htbp]
\centering
\begin{minipage}{0.48\textwidth}
  \centering  
  \includegraphics[width=0.83\textwidth]{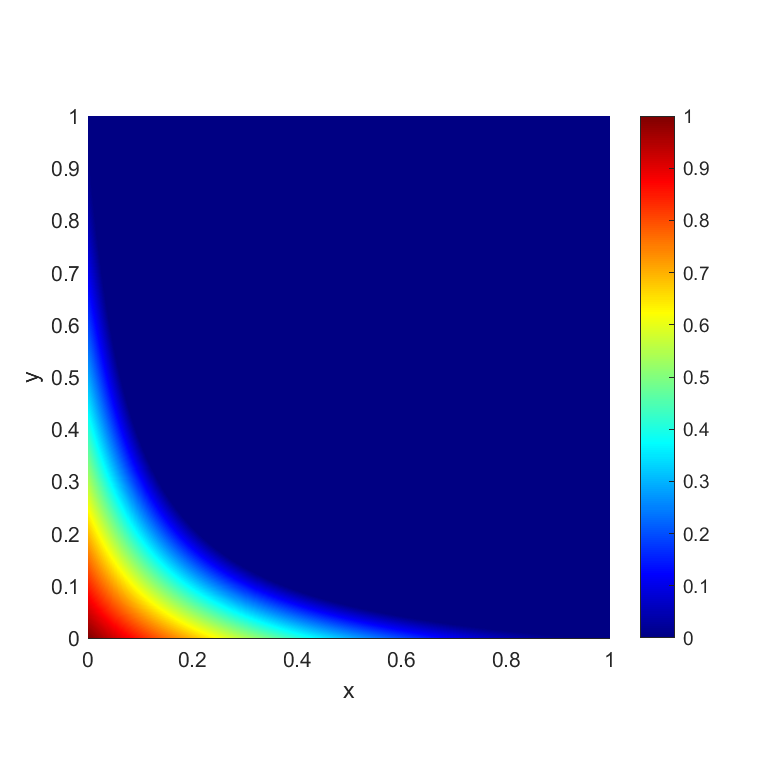}
  \caption{Exact solution $u^*$ in Example 4.}
  \label{fig_ex1ex_L2err}
\end{minipage}\hspace{-37.5pt}
\begin{minipage}{0.48\textwidth}
  \centering
  \includegraphics[width=0.83\textwidth]{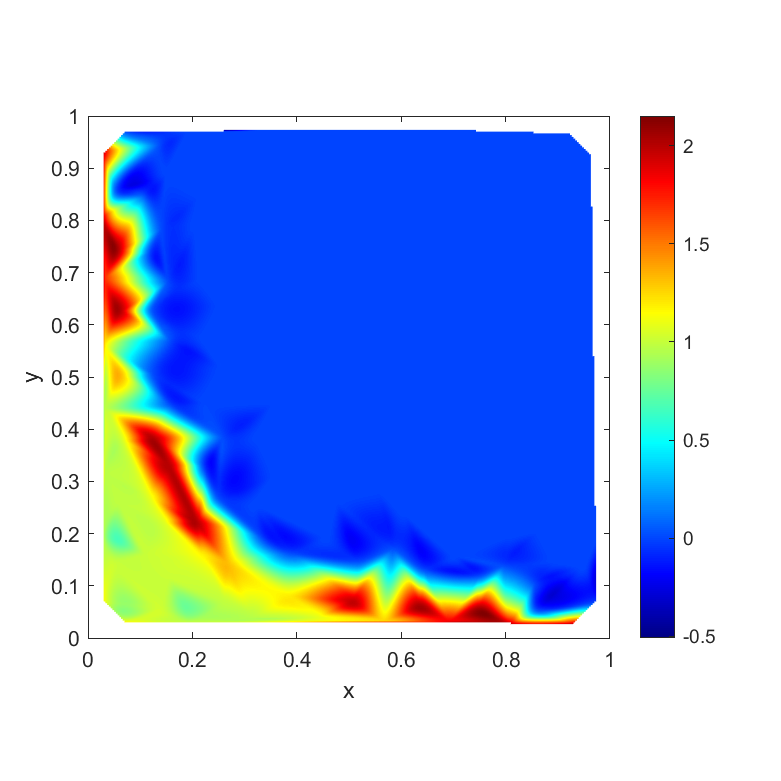}
  \caption{$e_{K_u}$ of Example 4. Take $h_u=h=1/8.$}
  \label{fig_exp11_L2err}
\end{minipage}
\end{figure}
\begin{figure}[H]
\centering
\begin{minipage}{0.48\textwidth}
  \centering  
  \includegraphics[width=0.83\textwidth]{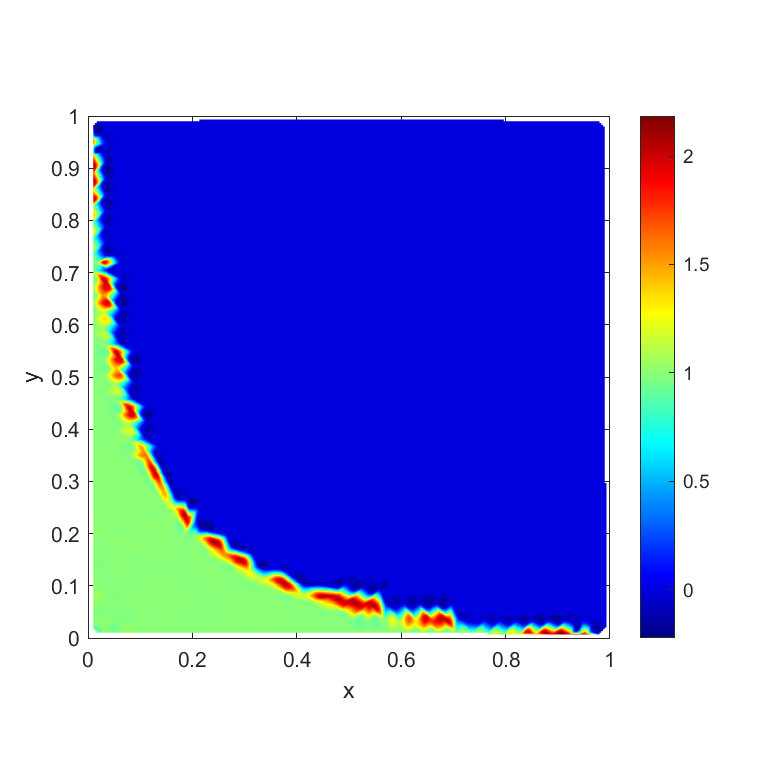}
  \caption{$e_{K_u}$ of Example 4. Take $h_u=h=1/32.$}
  \label{fig_exp13_L2err}
\end{minipage}\hspace{-37.5pt}
\begin{minipage}{0.48\textwidth}
  \centering
  \includegraphics[width=0.83\textwidth]{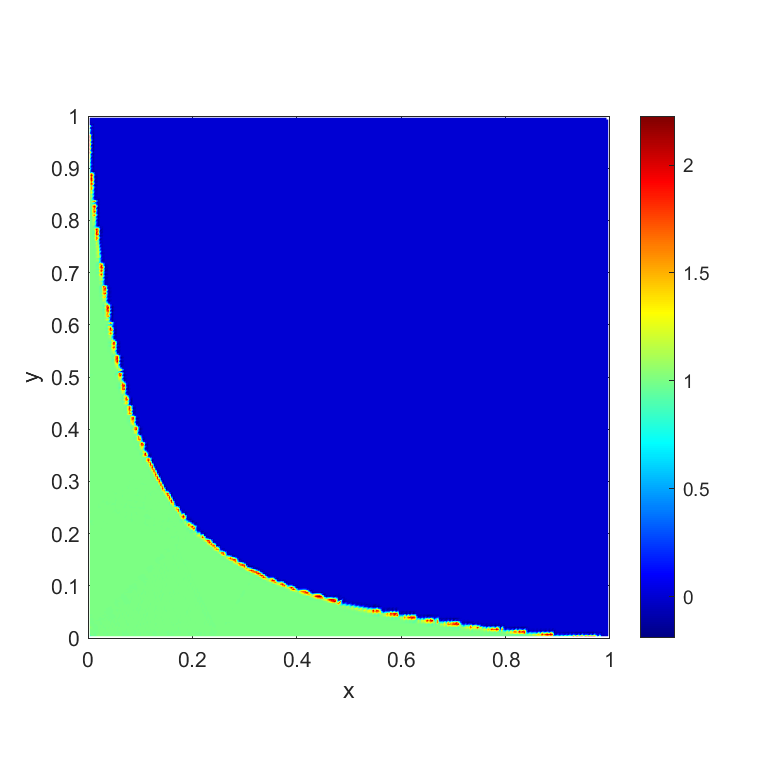}
  \caption{$e_{K_u}$ of Example 4. Take $h_u=h=1/128.$}
  \label{fig_exp15_L2err}
\end{minipage}
\end{figure}

\noindent\textbf{Example 5.} 
We will present the results of 
$e_{K_y},e_{K_p}$ and $e_{K_u}$ for the problem in Example 2. 
We can see that the maximum and minimum values 
of $e_{K_y}$ 
and $e_{K_p}$ show no significant variation 
with the refinement of $\mathcal{T}_h$.
\begin{table}[H]
\centering
\renewcommand\arraystretch{1.0}
\begin{tabular}{l|l|l|l|l|l|l}
\hline
$h$&1/4&1/8&1/16&1/32&1/64&1/128\\
\hline
$\max_K e_{K_y}$&9.651&18.10&20.88&24.28&26.00&26.60\\
\hline
$\min_K e_{K_y}$&2.376&2.313&2.111&1.964&1.807&1.620\\
\hline
$\max_K e_{K_p}$&9.841&18.03&21.50&23.92&24.56&23.99\\
\hline
$\min_K e_{K_p}$&1.407&1.294&1.285&1.333&1.222&1.093\\
\hline
\end{tabular}
\caption{$e_{K_y}$ and $e_{K_p}$ of Example 5. Take $m=1,h_u=h.$}
\end{table}
Next we show the results of $e_{K_u}$ for $m=1$.  Consistent with Example 4, 
we plot $\min\{e_{K_u}, 4\}$ to avoid the influence of extreme 
values in $e_{K_u}$. In this case, our restriction is $U_{ad}=\{
u\in L^2(\Omega_u)\mid0\leq u\leq 1\ \mathrm{a.e.\ in}\ 
\Omega_u\}$. Two borderlines are evident in the results, 
corresponding to the conditions $\{u_d-s(p^*)=0\}$ and $\{u_d
-s(p^*)=1\}$. This is illustrated in Figures \ref{fig_ex3ex_L2err} and \ref{fig_exp35_L2err}. Outside these boundaries,
$e_{K_u}$ remains under 4.

\begin{figure}[H]
\centering
\begin{minipage}{0.48\textwidth}
  \centering  
  \includegraphics[width=0.83\textwidth]{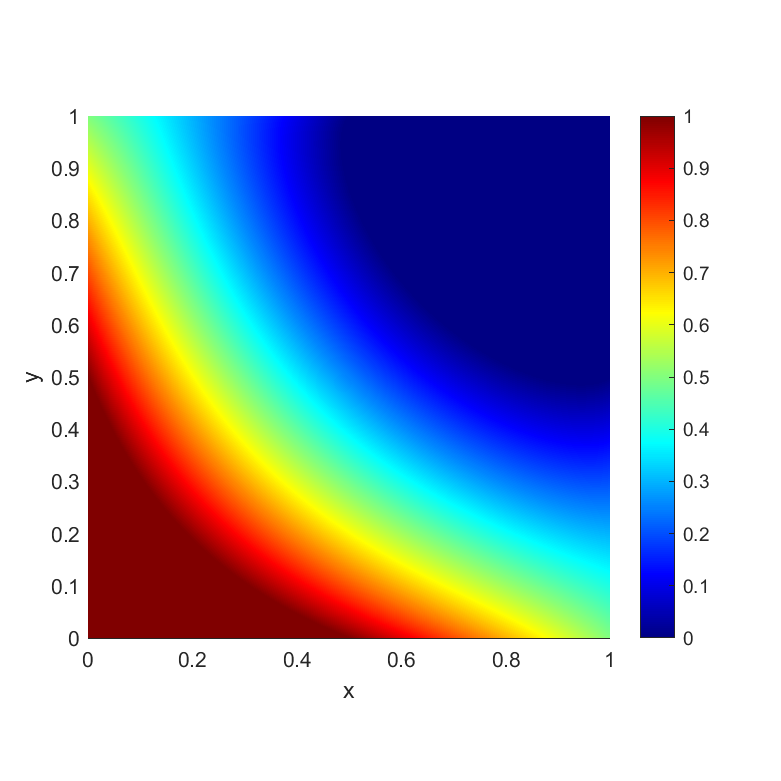}
  \caption{Exact solution $u^*$ in Example 5.}
  \label{fig_ex3ex_L2err}
\end{minipage}\hspace{-37.5pt}
\begin{minipage}{0.48\textwidth}
  \centering
  \includegraphics[width=0.83\textwidth]{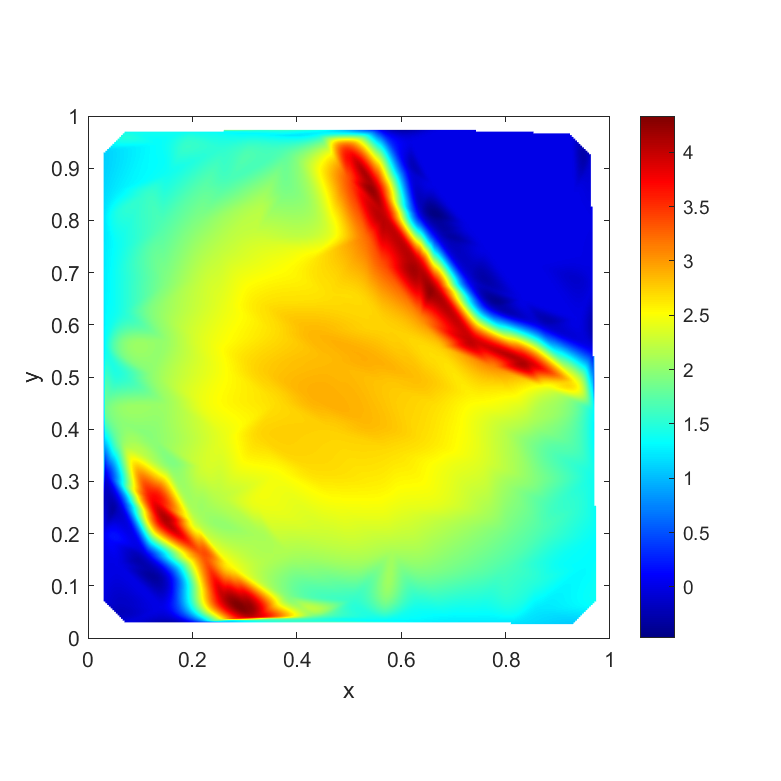}
  \caption{$e_{K_u}$ of Example 5. Take $h_u=h=1/8.$}
  \label{fig_exp31_L2err}
\end{minipage}
\end{figure}
\begin{figure}[H]
\centering
\begin{minipage}{0.48\textwidth}
  \centering  
  \includegraphics[width=0.83\textwidth]{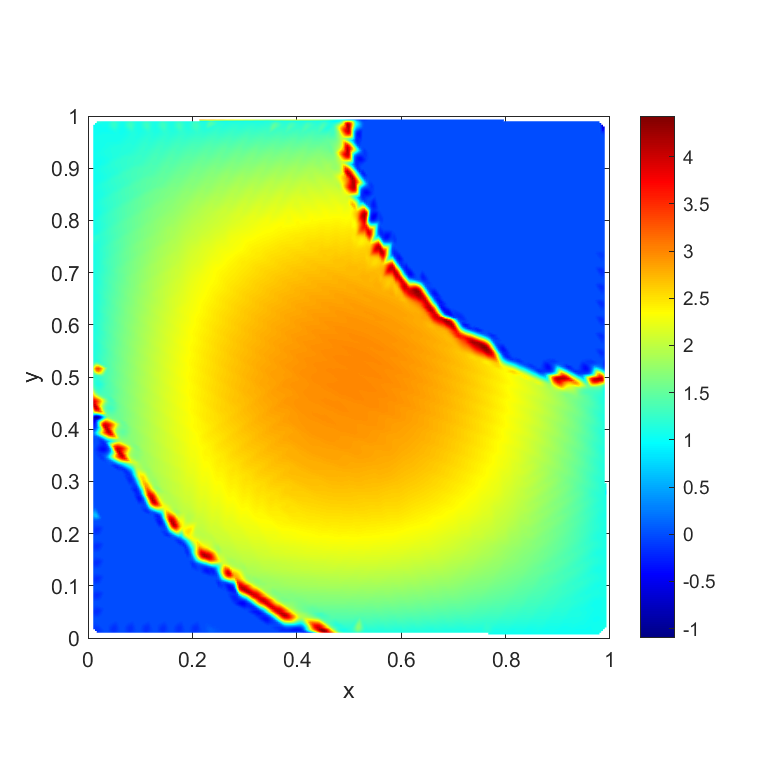}
  \caption{$e_{K_u}$ of Example 5. Take $h_u=h=1/32.$}
  \label{fig_exp33_L2err}
\end{minipage}\hspace{-37.5pt}
\begin{minipage}{0.48\textwidth}
  \centering
  \includegraphics[width=0.83\textwidth]{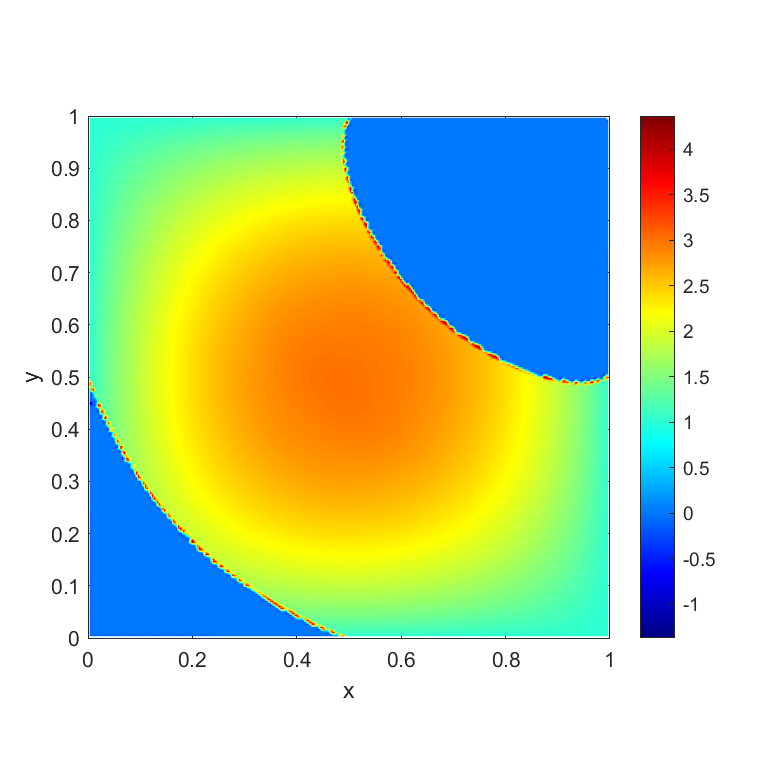}
  \caption{$e_{K_u}$ of Example 5. Take $h_u=h=1/128.$}
  \label{fig_exp35_L2err}
\end{minipage}
\end{figure}
\section{Conclusion}\label{sec6}
This work applies the reconstructed discontinuous approximation 
method to distributed elliptic 
optimal control problems with a strongly convex cost 
functional. We establish well-posedness of the discrete 
problem and derive both a priori and a posteriori error 
estimates in the $L^2$-norm and the energy norms. The discrete 
system is solved via a projected gradient descent algorithm 
with linear convergence rate. Numerical 
results substantiate the theoretical estimates and demonstrate 
the method's practical effectiveness.

\section*{Acknowledgments}
The research of this project is partially funded by National Natural 
Science Foundation of China (No. 12288101 and 12401399). The 
computational resources are partially supported by High-performance 
Computing Platform of Peking University.

\bibliographystyle{amsplain}
\bibliography{ref}

\end{document}